\documentclass[12pt]{amsart}
\usepackage[all, cmtip]{xy}

\usepackage{times,amsfonts,amsmath,amstext,amsbsy,amssymb, amsopn,amsthm,upref,eucal, mathrsfs}

\newcommand{\clspan}{\mathop{\overline{\mathrm{span}}}\nolimits}

\newcommand \la {\lambda}

\newcommand \GG {\mathscr G}

\newcommand \Conf {{\mathrm {Conf}}}

\newcommand \PW {{\mathcal {PW}}}

\newcommand \id {{\mathrm {id}}}

\newcommand \tr {{\mathrm {tr}}}

\newcommand{\TT}{\mathfrak T}

\newcommand \Prob {{\mathbb P}}

\newcommand{\SIN}{\mathscr S}
\newcommand{\expsin}{\mathbb E_{{\mathbb P}_{\mathscr S}}}

\newcommand{\probsin}{{{\mathbb P}_{\mathscr S}}}

\newcommand{\ee}{\mathbb E}
\newcommand{\ZZ}{\mathbb Z}
\newcommand{\NN}{\mathbb N}
\newcommand{\RR}{\mathbb R}

\newcommand{\R}{\mathbb R}

\newcommand{\eps}{\varepsilon}
\newcommand{\onehalf}{{H_{{1}/{2}}}}

\newcommand{\norm}{\mathscr{N}}

\newcommand{\BH}{\mathscr {BH}}

\newcommand {\QQ} {\mathscr Q}

\newcommand {\QB} {\mathscr {QB}}

\newcommand {\BB} {\mathscr B}
\newcommand {\qmax} {q_{\mathrm{max}}}

\newcommand \psitwo {\Psi^{(2)}}

\newtheorem{theorem}{Theorem}[section]

\newtheorem{lemma}[theorem]{Lemma}

\newtheorem{corollary}[theorem]{Corollary}
\newtheorem{proposition}[theorem]{Proposition}
\newtheorem{conjecture}[theorem]{Conjecture}

\def\det{\qopname\relax o{det}}
\title{The sine-process  has excess one}
\author{Alexander I. Bufetov}
\address{
Aix-Marseille Universit{\'e}, Centrale Marseille, CNRS, Institut de Math{\'e}matiques de Marseille, UMR7373, 
  39 Rue F. Joliot Curie 13453, Marseille, France;\newline
  Steklov  Mathematical Institute of RAS, Moscow, Russia;\newline
    Institute for Information Transmission Problems, Moscow, Russia.  }
  \email{alexander.bufetov@univ-amu.fr}

\begin{document}
\maketitle
\begin{abstract} 
The main result of this paper is that  almost every realization of the sine-process with one particle removed is a uniqueness set for the Paley-Wiener space; with two particles removed,  a zero set for the Paley-Wiener space.
\end{abstract} 
\tableofcontents
\section{Introduction}
\subsection{Formulation of the main result}
Let 
$
\SIN(x,y)=\displaystyle \frac{\sin \pi(x-y)}{\pi(x-y)}
$
be the sine-kernel, the kernel of the orthogonal projection in $L_2(\RR)$ onto the Paley-Wiener 
space $\PW$ of square-integrable functions whose Fourier transform is supported in $[-\pi, \pi]$.
The sine-kernel $\SIN$ induces  a determinantal measure $\probsin$ on the space $\Conf(\R)$ of configurations on $\R$;   the precise definitions are recalled in the next section. Ghosh \cite{ghosh-compl} proved 
that  $\probsin$-almost any  configuration $X\in\Conf(\R)$ is a uniqueness set for the Paley-Wiener 
space $\PW$: if $f\in\PW$ vanishes at every point $x\in X$, then $f=0$ identically.

The main result of this paper is that $\probsin$-almost every configuration {\it with one particle removed} is still a uniqueness set for  the Paley-Wiener 
space $\PW$; while $\probsin$-almost every configuration with two particles removed is a zero set for a non-zero Paley-Wiener function.
\begin{theorem}\label{mainthm}
For $\probsin$-almost every $X\in\Conf(\R)$ we have 
\begin{enumerate}
\item for any $p\in X$, if $f\in\PW$ satisfies $f(x)=0$ for all $x\in X\setminus p$, then $f=0$ identically; 
\item for any $p, q\in X$, $p\neq q$, the function $G_X^{p,q}$ given by the formula
$$
G_X^{p,q}(t)=\lim\limits_{ n\to\infty}\prod_{x\in X\setminus \{p,q\}, |x|<n^4} \left(1-\frac tx\right)
$$
is well-defined and belongs to the Paley-Wiener space $\PW$.
\end{enumerate}
\end{theorem}
The uniqueness property
for general determinantal point processes induced by orthogonal projections was conjectured to hold by Lyons and Peres and established in \cite{BQS}.
An analogue of the second statement of Theorem \ref{mainthm} holds for  determinantal point processes induced by one-dimensional integrable kernels satisfying a growth condition at infinity, see Subsection \ref{int-ker} below.
\begin{conjecture}\label{main-conj}
For $\probsin$-almost every configuration $X$ and any particle $p\in X$, the set $X\setminus p$ is hereditary complete.
\end{conjecture}
 It is tempting to try to prove the hereditary completeness using the method of Baranov, Belov and Borichev \cite{BBB}.
 
To a configuration $X$ on $\R$ such that the series 
$\sum x^{-2}$ converges and the limit $$\lim\limits_{ n\to\infty}\sum_{x\in X, |x|<n^4}  x^{-1}$$ exists, and, in particular, cf. \cite{buf-cond}, to $\probsin$-almost every configuration, one can assign the entire  
function 
$$
G_X(t)=\lim\limits_{ n\to\infty}\prod_{x\in X, |x|<n^4} \left(1-\frac tx\right).
$$
The  first statement of Theorem  \ref{mainthm} is a corollary of 
\begin{lemma}\label{main-lemma}
For $\probsin$-almost every configuration $X$ we have 
 $$
 \frac{G_X(t)}{\sqrt{1+t^2}} \notin L_2(\R).
 $$
\end{lemma}
The second statement of Theorem \ref{mainthm} is equivalent,  for $\probsin$-almost every configuration $X$,  to the relation
 $$\frac{G_X(t)}{1+t^2} \in L_2(\R).$$
 The proof of the second statement of Theorem \ref{mainthm} is given in Subsection \ref{int-ker} below.
 Derivation of the first statement of Theorem \ref{mainthm} from Lemma \ref{main-lemma} is given in Subsection \ref{der-thm-lem} below.

If $L$ is a Hilbert space and $(v_n)$, $n\in\mathbb N$, a family of vectors in $L$, then we write $\clspan^L(v_n)_{n\in\mathbb N}$ for the closure of the set of all finite linear combinations of $v_n$ and say that the family $(v_n)$ is \emph{complete} if $L=\clspan^L(v_n)_{n\in\mathbb N}$ and \emph{minimal} if $v_i\notin\clspan^L(v_n)_{n\in\mathbb N\setminus \{i\}}$ for all $i\in\mathbb N$.

Theorem \ref{mainthm}, equivalently, states that, for $\probsin$-almost every configuration $X$ and any particle 
$p\in X$ the sequence  of reproducing kernels
$$
\frac{\sin \pi(\cdot-x)}{\pi(\cdot-x)}, \ x\in X\setminus p
$$
is a complete minimal system in $\PW$, or, equivalently again, that, for $\probsin$-almost every configuration $X$ and any particle 
$p\in X$ the sequence  of complex exponentials 
$$
\exp(i x t ), \ x\in X\setminus p, \ t\in [-\pi, \pi]
$$
is a complete minimal system in $L_2([-\pi, \pi])$.

For a configuration $X\in \Conf(\RR)$ and a bounded subset $B$ of $\RR$, let the symbol $\#_B(X)$ stand for the number of the particles of X in $B$.
Let $I\subset \RR$ be a compact interval.   Given a configuration $X\in \Conf(\RR)$ 
, let 
$\mathrm{codim} (\PW; X; \RR\setminus I)$ be the codimension  of the closed span 
of the family of reproducing kernels $\SIN(\cdot, x)$, $x\in X\cap (\RR\setminus I)$ in the space $\PW$.
Theorem \ref{mainthm} implies 
 \begin{corollary}\label{quant-rig}
 For any fixed compact interval $I\subset \RR$, for   $\probsin$- almost all $X\in \Conf(\RR)$  we have
$$
\#_I(X)=1+\mathrm{codim} (\PW; X; \RR\setminus I).
$$
\end{corollary}
\subsection{Outline of the argument}
Take $d\in\NN$, $A>1$,  and introduce the function 
\begin{equation}\label{f-hat}
f^{d,A}(t)=\displaystyle \int_0^1 \log ((t+d+u)^2+A^2)du -  \log (t^2+A^2).
\end{equation}
Given a function $f$ on $\RR$, the corresponding additive functional $S_f$ is defined on the space of 
configurations $\Conf(\R)$ by the formula
$$
S_{f}(X)=\sum\limits_{x\in X} f(x).
$$
The  additive functional is first only defined when the sum in the right-hand side converges absolutely. 
The definition is then extended onto a larger family of functions: for example,  cf. \cite{buf-cond}, if $f$ is a bounded continuous function such that the limit $\lim\limits_{|x|\to\infty} xf(x)$ exists and is finite, then 
 the sum in the right-hand side converges, $\probsin$-almost surely along the 
subsequence  $\{x: |x|<n^4\}$. Lemma \ref{main-lemma}  is  derived from 
\begin{lemma}\label{main-lemma-bis}
For any natural $m\geq 3$ there exist constants $a_0, A_0$,  depending only on $m$, satisfying $1<a_0<A_0$, such that  for all $A\in (a_0, A_0)$ the following holds.
For any $\theta\in (0, 2\sqrt{2})$ and any $\delta>0$ we have
\begin{equation}\label{nonint-pos-dens}
\limsup\limits_{T\to\infty}\probsin\left(\left\{X\in \Conf(\RR):
\sum \limits_{t=T}^{2T} \mathbb I_{\{S_{f^{t,A}}(X)>\theta\log T\}} >T^{1-\theta^2/8-\delta}\right\}\right)>0.
\end{equation}
\end{lemma}
The derivation of Lemma \ref{main-lemma}   from  Lemma \ref{main-lemma-bis} is given in Subsection \ref{der-lem-lem}.
One of the difficulties in proving Lemma \ref{main-lemma} is that the 
events 
\begin{equation}\label{lg-dev1}
{\{S_{f^{t,A}}(X)>\theta\log T\}}
\end{equation}
 are very strongly correlated. Following the scheme used by Kistler \cite{kistler}, Arguin, Belius and Bourgade \cite{arguin}, we therefore decompose the events \eqref{lg-dev1}  into smaller events with the property of {\it hierarchical independence} by introducing a decomposition on the logarithmic scale in the frequency space.

Our convention for the Fourier transform is 
\begin{equation}\label{fourier-def}
\hat f(\la)=\frac 1{2\pi}\int\limits_{\RR} f(t) \exp(-i\la t) dt,
\end{equation}
so that 
\begin{equation}\label{fourier-inv-def}
f(t)=\int\limits_{\RR} \hat f(\la) \exp(i\la t) d\la.
\end{equation}

Introduce the Hardy spaces  $$
\mathcal H_2^+=\{\varphi\in L_2: \widehat \varphi|_{(-\infty, 0)}=0\}, 
\mathcal H_2^-
=\{\varphi\in L_2: \widehat \varphi|_{(0, +\infty)}=0\},
$$
and let $\mathscr P_+$, $\mathscr P_-$ be the corresponding projection operators.

Take a natural $T$; we will later let $T$ go to infinity.  
For  a natural $d\in [T, 2T]$ write 
$$
\phi^{d, A}(t)=\log((d+t)^2+A^2)-\log(t^2+A^2).
$$
Set
$$
 \phi^{d, A}_+=\mathscr P_+ \phi^{d, A}, \ \phi^{d, A}_-=\mathscr P_- \phi^{d, A}.
 $$
 Given a set $Y$, here and below the symbol $\mathbb  I_{Y}$ will stand for the indicator function of the set $Y$.
 For the Fourier transforms  we have 
 \begin{equation}\label{fourier-phi-plus}
 \widehat{\phi^{d, A}_+}(\la)=\frac{\exp(i\la d)-1}{\la}\exp(-A\la)\mathbb  I_{(0, +\infty)},
 \end{equation}
 and 
 \begin{equation}\label{fourier-phi}
 \widehat{\phi^{d, A}}(\la)=\frac{\exp(i\la d)-1}{|\la|}\exp(-A|\la|).
 \end{equation}
 For the Fourier transform of the function 
$f^{d,A}(t)$  given by \eqref{f-hat},
we have 
 \begin{equation}\label{fourier-f-plus}
 \widehat{f^{d, A}_+}(\la)=\frac{(i\la)^{-1}(\exp(i\la)-1)(\exp(i\la d)-1}{\la}\exp(-A\la)\mathbb  I_{(0, +\infty)}.
 \end{equation}
 \begin{equation}\label{fourier-f}
 \widehat{f^{d, A}}(\la)=\frac{(i\la)^{-1}(\exp(i\la)-1)\exp(i\la d)-1}{|\la|}\exp(-A|\la|).
 \end{equation}
 
Introduce the function $F^T_+\in \mathcal H_2^+$ by setting  
$$
\widehat {F^T_+}=\frac 1{\la}\chi_{[T^{-1}, 1]}.
$$
Set $F^T_-=\overline {F^T_+}$, $F^T_+=F^T_++F^T_-$.
Take $m\in \NN$, $m\geq 3$. 
Take $l=1, \dots, m-1$ and set 
$$
f^{d,A,l}_ +=(\mathbb I_{[T^{-l/m}, T^{(1-l)/m}]}( f^{d, A}_ + -F^T_+) \ {\widehat{}}\ ) \ {\check{}};
$$
$$
f^{d,A,m}_ +=f^{d,A}_ +-\sum\limits_{l=1}^{m-1}f^{d,A,l}_ +.
$$
Set $f^{d,A,l}_-=\overline{f^{d,A,l}_+}$, $f^{d,A,l}=f^{d,A,l}_++ f^{d,A,l}_-$.
We now consider the events $V_d$, $d=T, T+1, \dots, 2T$, given by the formula
$$
V_d=\left\{X\in \Conf(\RR): S_{F^T}(X)\geq \frac{\theta \log T}2, S_{f^{d,A,l}}(X)\geq \frac{\theta \log T}{2(m-1)}, l\in [1, m-1] \right\}.
$$
We also introduce the set $W$ by the formula
$$
W=\left\{X:  S_{f^{d,A,m}}(X)\leq \frac{3\log T}{m-1}, d=T, \dots, 2T \right\}.
$$
For  the probability of the set $W$, we have 
\begin{lemma}\label{main-est-w}
For any natural $m\geq 3$ there exist constants $a_0, A_0$,  depending only on $m$, satisfying $1<a_0<A_0$, and a constant $\delta>0$ such that  for all $A\in (a_0, A_0)$ and 
all sufficiently large $T$ we have 
$$
\probsin(W)>1-T^{-\delta}.
$$
\end{lemma}

Lemma \ref{main-lemma-bis} directly follows from
\begin{lemma}\label{main-lemma-tris}
For any natural $m\geq 3$ there exist constants $a_0, A_0$,  depending only on $m$, satisfying $1<a_0<A_0$, such that  for all $A\in (a_0, A_0)$ the following holds.
For any $\theta\in (0, 2\sqrt{2})$ and any $\varepsilon>0$ we have
$$
\limsup\limits_{T\to\infty}\probsin\left(\left\{X\in \Conf(\RR):\sum\limits_{d=T}^{2T}
 {\mathbb I}_{V_d\cap W}(X)>\gamma T^{1-\theta^2/8-\varepsilon}\right\}\right)>0.
$$
\end{lemma}

For $t\in \mathbb R$, let $\norm(t, \sigma^2)$ stand for the probability that a Gaussian random variable with expectation zero and variance $\sigma^2$ is greater than $t$:
\begin{equation}\label{gaussian-tail}
{\norm}(t; \sigma^2)=\frac 1{\sqrt{2\pi}\sigma}\displaystyle \int\limits_s^{\infty} \exp(-s^2/2\sigma^2)dt.
\end{equation}

In order to prove Lemma \ref{main-lemma-tris}, and, consequently,  Lemma \ref{main-lemma-bis}, we establish the following estimates.

\begin{lemma}\label{main-est-vl}
For any natural $m\geq 3$ there exist constants $a_0, A_0$,  depending only on $m$, satisfying $1<a_0<A_0$,  and 
for  any $\theta_0>0$ there exist constants $c>0, C>0, T_0>0$ such that  
 for all $A\in (a_0, A_0)$, for all $\theta\in (0, \theta_0)$ and   all  $T>T_0$  the following holds. 
\begin{enumerate} 
\item For any $d\in [T, 2T]$ we have 
$$
c<\displaystyle \frac{\probsin(V_d)}{\norm\left(\theta\log T/2, 2\log T\right)\left(\norm\left(\frac{\theta \log T}{2(m-1)}, \frac{2\log T}{m-1}\right)\right)^{m-1}}<C.
$$
\item For  any $l=1, \dots, m-1$ and any $d_1, d_2\in [T, 2T]$ satisfying
$$
T^{l/m}\leq |d_1-d_2|
$$ 
we have 
$$
c<\displaystyle \frac{\probsin(V_{d_1}\cap V_{d_2})}{\norm(\theta\log T/2, 2\log T)\left(\norm\left(\frac{\theta \log T}{2(m-1)}, \frac{2\log T}{m-1}\right)\right)^{m+l-1}}<C.
$$
\end{enumerate} 
\end{lemma}

We shall prove these estimates in a slightly stronger form, see Corollary \ref{indep-prob-cor-two} below.
Once Lemmata  \ref{main-est-w}, \ref{main-est-vl}   are established, 
Lemma \ref{main-lemma-bis} directly  follows from a variant of the Paley-Zygmund inequality, see Lemma \ref{hier-ind} below.

The proofs of Lemmata \ref{main-est-w}, \ref{main-est-vl}   rely on the following 
estimates on the joint exponential moments of the random variables $S_{f^{d,A,l}}$.
We start with an exponential upper estimate for high frequencies.
For a complex vector $\vec a=(a_1, \dots, a_n)$, we write $|\vec a|= |a_1|+\dots+|a_n|$.
Take $l\in \{1, \dots, m-1\}$. Let $d_1, d_2$ be such that 
\begin{equation}\label{indep-dist}
T^{l/m}\leq |d_1-d_2|.
\end{equation}
Let $\la\in \mathbb C$, $\vec\la^{(1)}\in \mathbb C^{m-1}, \vec \la^{(2)}\in \mathbb C^l$. 
Write 
$$
f^{d_1, d_2, A, T}(\la, \vec\la^{(1)},\vec\la^{(2)})=\la F^{T,A}+\sum\limits_{r=1}^{m-1} \la^{(1)}_rf^{d_1,A,r}+
\sum\limits_{s=1}^{l}\la^{(2)}_sf^{d_2,A,s}.
$$
\begin{lemma}\label{indep-lemma-highfreq}
 For any natural $m\geq 3$ there exist constants $a_0, A_0$, depending only on $m$, satisfying $1<a_0<A_0$,
 such that  for all $A\in (a_0, A_0)$ the following holds.
For any $D>0$,   there exists $R>0$, $T_0>0$ such that for all $T>T_0$, all $A\in (a_0, A_0)$ and all  
$\la\in \mathbb C, \vec \la^{(1)}\in \mathbb C^{m-1}, \vec \la^{(2)}\in \mathbb C^l$ satisfying 
$$|\la|+|\vec \la^{(1)}|+|\vec \la^{(2)}|>R
$$ 
we have
$$
\expsin |\exp(S_{f^{d_1,d_2, A, T}(\la, \vec\la^{(1)},\vec\la^{(2)})})|\leq T^{-D}.
$$
\end{lemma}
For low frequencies, we establish subnormal estimates in the following precise sense.  
Let the symbol $c_{d_1, d_2}^{A, T}$ stand for the ratio of the joint exponential moment of our random variables and of independent Gaussians with corresponding variances:  
\begin{multline}\label{ctt-def}
c_{d_1, d_2}^{A, T}(\la, \vec \la^{(1)},\vec \la^{(2)})=\\=\exp\left(-\la^2\log T-\frac{\log  T}{m}\left(\sum\limits_{r=1}^{m-1} (\la^{(1)}_r)^2-\sum\limits_{s=1}^{l}(\la^{(2)}_s)^2\right)\right)\times \\ \times 
\expsin \exp\left(S_{f^{d_1, d_2, A, T}(\la, \vec\la^{(1)},\vec\la^{(2)})}\right). 
\end{multline}
\begin{lemma}\label{indep-lemma-lowfreq}
For any natural $m\geq 3$ there exist constants $a_0, A_0$, 
depending only on $m$, satisfying $1<a_0<A_0$, such that  for all $A\in (a_0, A_0)$ the following holds.
\begin{enumerate}
\item 
 For any $R>0$ there exists $B>0$ such that for all $T>1$, all  $A\in (1, A_0)$ and all
$\la\in \mathbb C, \vec \la^{(1)}\in \mathbb C^{m-1}, \vec \la^{(2)}\in \mathbb C^l$ 
satisfying $$|\la|+|\vec \la^{(1)}|+|\vec \la^{(2)}|\leq R$$ we have
$$ 
\|\mathrm{grad} \ c_{d_1, d_2}^{A, T}(\la, \vec \la^{(1)},\vec \la^{(2)})\|\leq B.
$$
\item There exists a positive constant $c_0=c_0(m)$ depending only on $m$ 
 such that  for any $R>0$, all sufficiently large  $T$, all 
$A\in (a_0, A_0)$ and all $\theta\in \mathbb C, \vec \theta^{(1)}\in \mathbb C^{m-1}, \vec \theta^{(2)}\in \mathbb C^l$ 
satisfying $$|\theta|+|\vec \theta^{(1)}|+|\vec \theta^{(2)}|\leq R$$ we have 
$$
c_{d_1, d_2}^{A, T}(\theta, \vec \theta^{(1)}, \vec \theta^{(2)})\geq c_0.
$$
\end{enumerate}
\end{lemma}
By  the Parseval identity, the estimates given in Lemmata \ref{indep-lemma-highfreq}, \ref{indep-lemma-lowfreq} 
imply the desired estimates on probabilities of large deviations, 
cf. Lemma \ref{subnorm-lower} below. 
It remains to establish Lemmata \ref{indep-lemma-highfreq}, \ref{indep-lemma-lowfreq}. This will be achieved in the ensuing sections.
\subsection{Further directions}
Applying the general scheme of Kistler \cite{kistler}, Arguin -- Belius -- Bourgade \cite{arguin}, 
from Lemmata \ref{main-lemma-bis}, \ref{main-lemma-tris} one derives, for some, and, consequently, for any, $A>0$,  
the logarithmic asymptotic for the growth of the maximum of $G_{X+iA}(t+iA)$: indeed,  $\probsin$-almost surely we have 
\begin{equation}\label{max-A}
\limsup\limits_{T\to\infty}  \frac{\max\limits_{t: |t|<T} \log|G_{X+iA}(t+iA)|}{\log T}=\sqrt{2}.
\end{equation}
It would be interesting to derive the same asymptotic also for $A=0$:
\begin{conjecture}
The asymptotic formula \eqref{max-A} holds also for $A=0$.
\end{conjecture}
Similarly, one obtains an asymptotic formula for the growth of the $L_p$-norm of the function 
$G_{X+iA}(t+iA)$: for $p\in (0, \sqrt{2})$,  we have the $\probsin$-almost sure equality
\begin{equation}\label{smallp-norm-A}
\limsup\limits_{T\to\infty}  \displaystyle\frac{\log \displaystyle\int\limits_{0}^T |G_{X+iA}(t+iA)|^p}{\log T}=1+\frac{p^2}{2},
\end{equation}
whereas,  for $p>\sqrt{2}$, we have the $\probsin$-almost sure equality
\begin{equation}\label{largep-norm-A}
\limsup\limits_{T\to\infty} \displaystyle \frac{\log \displaystyle  \int\limits_{0}^T |G_{X+iA}(t+iA)|^p}{\log T}=\sqrt{2}p.
\end{equation}
Again, it would be interesting to show that the same asymptotic formulas hold for $A=0$. 
Finally, it seems natural to conjecture that, for small $p$, the function  $G_{X+iA}(t+iA)$ 
converges to a random measure:
\begin{conjecture}
\begin{enumerate}
\item
Let $p\in (0, \sqrt{2})$. The  random probability measure  defined on 
the unit interval $[0,1]$ by the formula
$$
\displaystyle\frac{|G_{X+iA}(T+Tt+iA)|^pdt}{\displaystyle\int\limits_{0}^1 |G_{X+iA}(T+Tt+iA)|^pdt}
$$
 converges, as $T\to\infty$, in law in the space of Radon measures on the unit interval endowed with the usual weak topology.
 \item 
 For $A=0$, the limit is a nontrivial Gaussian Multiplicative Chaos measure.
 \end{enumerate}

\end{conjecture}

\noindent {\bf{Acknowledgements.}} 
I am deeply grateful to A. Basteri, S. Berezin, 

\noindent A. Borichev, C. Carminati, V. Kaimanovich, A. Klimenko, I. Krasovsky, 

\noindent A. Miftakhov, G. Olshanski, A. Papini, U. Pappalettera, Y. Qiu, R. Rhodes, I. Simunec,  
 C. Webb, D. Zubov for useful discussions.

Part of this work was done during visits to the Institut Mittag-Leffler of the Royal Swedish Academy of Sciences, the University of Bielefeld, the University of Cagliari, the University of Copenhagen,  the University of Kyushu, the University of Pisa and the University of Rome ``La Sapienza''.  I am deeply grateful to these institutions for their warm hospitality. 

This research received support from the European Research Council (ERC) under the European Union Horizon 2020 research and innovation programme, grant 647133 (ICHAOS), from the Agence Nationale de la Recherche, project ANR-18-CE40-0035, and from the Russian Foundation for Basic Research, grant 18-31-20031.

\section{Conditional measures of determinantal point processes}\label{appendix}

\subsection{Spaces of configurations.}

Let $E$ be a locally compact complete metric space.
A {\it configuration} on $E$ is a collection  of points in $E$, called {\it particles},  considered without regard to order and
subject to  the additional requirement that every bounded set contain only finitely many particles of a configuration.
Let $\Conf(E)$ be the space of configurations on $E$.
For a bounded Borel set $B\subset E$,   let
$$\#_B\colon\Conf(E)\to\mathbb{N}\cup\{0\}$$ be
the function that to a configuration
assigns the number of its particles
belonging to~$B$. The random variables $\#_B$ over all  bounded Borel sets $B\subset E$
determine the  Borel sigma-algebra on $\Conf(E)$.

Let $\varphi$ be a  measurable function on $E$, and introduce the corresponding {\it additive functional} 
 $S_{\varphi}$ on $\Conf(E)$ by the formula 
\begin{equation}\label{sn-def}
S_{\varphi}(X)=\sum\limits_{x\in X} \varphi(x).
\end{equation}
If the sum in the right-hand side fails to converge absolutely,
then the additive functional is not defined. In what follows, we will consider {\it regularized} additive functionals for determinantal point processes.

Let $g$ be a non-negative measurable function on $E$, and introduce the
{\it multiplicative functional} $\Psi_g:\Conf(E)\to\mathbb{R}$ by the formula
 \begin{equation} \label{mult-fun-def}
 \Psi_g(X)=\prod\limits_{x\in X}g(x).
 \end{equation}
If the infinite product
$\prod\limits_{x\in X}g(x)$ absolutely converges to $0$ or to $\infty$, then we set, respectively,
$\Psi_g(X)=0$ or $\Psi_g(X)=\infty$. If the product in the right-hand side fails to converge absolutely,
then the multiplicative functional is not defined. In what follows, we will consider {\it regularized} multiplicative functionals for determinantal point processes.

\subsection{Point processes.}
A Borel probability measure $\Prob$ on $\Conf(E)$ is called {\it a point process} with phase space $E$.
Recall that the point process $\Prob$ is said to admit correlation measures
of order $l$ if for any continuous compactly supported function $\varphi$ on $E^l$
the functional
$$
\sum\limits_{x_1, \dots, x_l\in X} \varphi(x_1, \dots, x_l)
$$
is $\Prob$-integrable; the sum is taken over all ordered $l$-tuples of distinct particles in $X$. The $l$-th correlation measure $\rho_l$ of the point process $\Prob$ is then
defined by the formula
$$
\ee_{\Prob} \left(\sum\limits_{x_1, \dots, x_l\in X} \varphi(x_1, \dots, x_l)\right)=
\displaystyle \int\limits_{E^l} \varphi(q_1, \dots, q_l)d\rho_l(q_1, \dots, q_l).
$$
In particular, taking  $l=1$ and a bounded compactly supported Borel function $\varphi: E\to\mathbb{R}$, we have $\ee_{\Prob} S_{\varphi} = \int\limits_E \varphi d\rho_1$. 

For a Borel $C\subset E$, the measure $\mathbb{P}(\cdot | X; C)$ on $\Conf(E\setminus C)$ is defined as the conditional measure of $\mathbb{P}$ with respect to the condition that the restriction of our random configuration onto $C$ coincide with $X\cap C$. Consider the surjective restriction mapping $X \to X\cap C$ from
$\textrm{Conf}(E)$ to $\textrm{Conf}(C)$. Fibres of this mapping are identified with $\mathrm{Conf}(E\backslash C)$, and conditional measures, in the sense of Rohlin \cite{Rohmes},  are the measures $\mathbb{P}(\cdot | X; C)$, cf. \cite{buf-aop}, \cite{buf-cond}.

\subsection{Campbell  and Palm Measures.}
Following Kallenberg \cite{kallenberg}, Daley--Vere-Jones \cite{DVJ}, we recall the definition of Campbell 
measures of point processes; the notation follows  \cite{buf-aop}.
Let $\Prob$ be a point process on $E$  admitting the first correlation measure 
$\rho_1^{\Prob}$.
The {\it Campbell measure}  ${\EuScript C}_{\Prob}$ of  $\Prob$ 
is a sigma-finite measure on $E\times \Conf(E)$ such that for any Borel subsets
$B\subset E$, ${\mathscr Z}\subset \Conf(E)$ we have
$$
{\EuScript C}_{\Prob}(B\times {\mathscr Z})=\displaystyle \int\limits_{{\mathscr Z}} \#_B(X)d\Prob(X).
$$
The Palm measure ${\hat \Prob}^q$ is the canonical conditional measure, in the sense of Rohlin \cite{Rohmes}, 
of the Campbell measure ${\mathcal C}_{\Prob}$
with respect to the measurable partition of the space $E\times \Conf(E)$
into subsets  $\{q\}\times \Conf(E)$, $q\in E$, cf. \cite{buf-aop}.
By definition, the Palm measure ${\hat \Prob}^{q}$ is supported on the subset of configurations containing a particle at  position $q$.
Removing these particles, one defines the {\it reduced}
Palm measure $\Prob^{q}$ as the push-forward of the Palm measure
${\hat \Prob}^{q}$ under the erasing map
$X\to X\setminus \{q\}$. Iterating the definition, one arrives at iterated Campbell, Palm and reduced Palm measures; 
the reduced Palm measure with respect to positions $q_1, \dots, q_l$ will be denoted $\Prob^{q_1, \dots, q_l}$; 
see Kallenberg \cite{kallenberg}, whose formalism is also adopted in  \cite{buf-aop}, for a more detailed exposition. 
As all conditional measures, reduced Palm measures $\Prob^q$ are {\it a priori} only defined  for $\rho_1$-almost every 
$q$. In our context of determinantal point processes, the Shirai-Takahashi Theorem recalled below will allow us to fix a convenient family $\Prob^p$, $p\in E$, of reduced Palm measures and say that 
a realization is chosen for the family of  reduced Palm measures.

\subsection{Determinantal Point Processes}
Let $\mu$ be a sigma-finite Borel  measure on $E$. We consider two cases: the continuous case when $\mu$ has no 
atoms and assigns positive weight to every open set, as well as 
the discrete case when $E$ is countable and $\mu$ is the counting measure. We always assume that $\mu$ assigns positive weight
to nonempty open sets.
Recall that a Borel probability measure $\mathbb{P}$ on
$\Conf(E)$ is called
\textit{determinantal} if there exists a locally trace class operator  $K$ acting in $L_2(E, \mu)$   such that for any bounded measurable
function $g$, for which $g-1$ is supported in a bounded set $B$,
we have
\begin{equation}
\label{eq1}
\mathbb{E}_{\mathbb{P}}\Psi_g
=\det\biggl(1+(g-1)K\chi_{B}\biggr).
\end{equation}
Here and elsewhere in similar formulas, $1$ stands for the identity operator.
The Fredholm determinant in~\eqref{eq1} is well-defined since
$K$ is locally of trace class.
The equation (\ref{eq1}) determines the measure $\Prob$ uniquely. We use the notation $\Prob_K$ for the determinantal measure 
induced by the operator $K$.
By a theorem due to Macch{\` \i} and Soshnikov ~\cite{Macchi}, \cite{soshnikov} and Shirai-Takahashi \cite{ShirTaka0}, any
Hermitian positive contraction that belongs
to the local trace class defines a determinantal point process.
Recall the well-known fact (see, e.g., (65) in \cite{buf-aop}) that if $\Pi$ is the kernel of a  a locally trace-class orthogonal projection acting
 in $L_2(E, \mu)$ and $\varphi$ a compactly supported 
bounded Borel function on $E$, then  the variance of the additive functional $S_{\varphi}$ is given by 
\begin{equation}\label{var-sphi}
\mathrm{Var}_{\Prob_{\Pi}} S_{\varphi}=\frac12\int\limits_{E}\int\limits_{E} \left|\varphi(t)-\varphi(s)\right|^2 |\Pi(s,t)|^2d\mu(s)d\mu(t).
\end{equation}

\subsection{Continuity of multiplicative functionals}
We  slightly  extend the definition of the Fredholm determinant. First recall that the Hilbert-Carleman regularization $\det_2$ of the Fredholm determinant is introduced on finite rank operators  
by the formula
$$
\det_2(1+A)=\exp(-\tr A) \det(1+A)
$$
and then extended by continuity onto all Hilbert-Schmidt operators.
Consider now a kernel $K$ on $E\times E$,  inducing a locally trace-class Hilbert-Schmidt operator. 
Convergence in principal value is understood as convergence along a sequence of a fixed, previously chosen, exhausting sequence $B_n$ of bounded subsets of $E$:
$$
\int\limits^{v.p.}_{E} f(t)d\mu(t)=
\lim\limits_{n\to\infty}\int\limits_{B_n} f(t)d\mu(t).
$$
If the integral in principal value
$$
\displaystyle \int^{v.p.}\limits_{E} K(x,x)d\mu(x)
$$
is well-defined,  then we set 
\begin{equation}\label{det-def}
\det(1+K)=\exp\left(\displaystyle \int\limits_{E} K(x,x)d\mu(x)\right)\det_2(1+K).
\end{equation}
Note here that the operator $K$ need not be trace-class. 

Given a locally trace class Hermitian kernel $\Pi$ on $E\times E$, let $L_2^{\Pi}(E, \mu)$ be the subspace of bounded functions $f$ such that the function $f(t)\Pi(t,t)$ is also bounded and the integral 
$$
\int\limits^{v.p.}_{E} f(t)\Pi(t,t)d\mu(t)
$$
is well-defined. From the definitions we directly have
\begin{lemma}\label{psif-cont}
Let $\Pi$ be the kernel of an orthogonal projection.
For any $p>0$, the correspondence  $f\to\Psi_{1+f}$ induces a continuous mapping from $L_2^{\Pi}(E, \mu)$ to 
$L_p(\Conf(E), \Prob_{\Pi})$. 
\end{lemma}
For the sine-process, the subspace  $L_2^{\SIN}(\RR)$ is simply the space $L_2^0(\RR)$ of bounded square-integrable functions such that the principal value integral 
$$
\int\limits^{v.p.}_{\RR} f(t)dt=
\lim\limits_{R\to\infty}\int\limits_{-R}^R f(t)dt
$$
is well-defined. Letting $H_1(\RR)$ be the usual Sobolev space of square-integrable functions with square-integrable derivative, cf. \eqref{sobolev-semi} below, write also $H_1^0(\RR)=H_1(\RR)\cap L_2^0(\RR)$.

\begin{corollary}\label{psif-cont-sine}
For any $p>0$, the correspondence  $f\to\Psi_{1+f}$ induces a continuous mapping from $L_2^0(\RR)$ to 
$L_p(\Conf(\RR), \probsin)$ and the correspondence $f\to\Psi_{1+if}$ induces a continuous mapping from $H_1^0(\RR)$ to 
$L_p(\Conf(\RR), \probsin)$.
\end{corollary}

\subsection{Palm Measures of Determinantal Point Processes.}
For $q\in E$ satisfying $\Pi(q,q)> 0$, introduce a kernel $\Pi^q$ by the formula
\begin{equation}\label{piqdef}
\Pi^q(x,y)=\Pi(x,y)-\displaystyle \frac{\Pi(x,q)\Pi(q,y)}{\Pi(q,q)},
\end{equation}
and, iterating, define the kernels $\Pi^{q_1, \dots, q_l}$.  
Shirai and Takahashi \cite{ShirTaka1} have proved
that for  any $l\in {\mathbb N}$ and for $\rho_l$-almost every $l$-tuple $q_1, \dots, q_l$ of distinct points in $E$,
the iterated reduced Palm measure $\Prob_{\Pi}^{q_1, \dots, q_l}$ is given by the formula
\begin{equation}\label{ppiq-it}
\Prob_\Pi^{q_1, \dots, q_l}= \Prob_{\Pi^{q_1, \dots, q_l}}.
\end{equation}
If $\Pi$ is an orthogonal projection onto a closed subspace $L\subset L_2(E, \mu)$, then
the kernel $\Pi^{q_1, \dots, q_l}$ induces an orthogonal projection onto the subspace
$$
L(q_1, \dots, q_l)=\{\varphi\in L: \varphi(q_1)=\ldots=\varphi(q_l)=0\}.
$$

\subsection{The Ghosh-Peres Rigidity} 
Given a Borel subset $C \subset E,$ we let $\mathcal{F}_C$ be the $\sigma$-algebra generated by all random variables of the form $\#_B, B\subset C.$ Write $\mathcal{F}_C^{\mathbb{P}}$ for the $\mathbb{P}$-completion of $\mathcal{F}_C$.
Following  Ghosh and Peres \cite{ghosh-compl}, \cite{GP}, \cite{G}, we say that a point process $\mathbb{P}$ on $E$ is  {\it rigid} if for any compact  subset $B \subset E$ the function $\#_B$ is  $\mathcal{F}_{E\backslash B}^{\mathbb{P}}$-measurable.

Rigidity implies that for  any precompact set $B\subset E$ and $\Prob$-almost any $X$ the conditional measure 
$\Prob(\cdot |X; E\setminus B)$ is supported on the subset of configurations containing precisely $l$ particles, where $l=\#_B(X)$.
If $\Pi$ is a self-adjoint kernel inducing a rigid  determinantal process, then the self-adjoint 
kernel $\Pi^{[\cdot |X; E\setminus B]}$, cf. \cite{BQS}, governing the conditional measure $\Prob(\cdot |X; E\setminus B)$, induces an orthogonal
 projection.

Ghosh \cite{G}, Ghosh and Peres \cite{GP}   established rigidity for the sine-process and for the Ginibre ensemble;  the approach
  of Ghosh and Peres has been followed in \cite{buf-rigid}, where rigidity is proven for 
 the determinantal point processes with the Airy, the Bessel and the Gamma kernels; in the stationary case, the criterion of Kolmogorov \cite{Kolmogorov}
  implies rigidity, cf. \cite{BDQ}.

\subsection{Conditional measures for the sine-process}
Theorem 1.4 in \cite{buf-aop}, Theorem 1.4 and Corollary 1.5 in \cite{buf-cond} give an explicit
 description, for determinantal point processes with integrable kernels,  
of conditional measures in a bounded domain with respect to fixing the configuration in the exterior.
We recall this description for the sine-process.
\begin{proposition}\label{rho-sine-gen}
Let $I$ be a compact interval on ${\mathbb R}$.
For $\Prob_{\mathscr S}$-almost any configuration $X\in \Conf({\mathbb R})$, the conditional measure $\mathbb{P}_{{\mathscr S}}(\cdot | X; {\mathbb R}\setminus I)$ has the form

\begin{equation}
Z(I,X)^{-1} \prod\limits_{1\leq i<j\leq \#_I(X)} (t_i-t_j)^2 
\prod \limits_{i=1}^{\#_I(X)} \prod\limits_{x\in X\setminus I} \left( 1-\frac {t_i}x\right)^2,
\end{equation}

\begin{equation}
Z(I,X)^{-1} \prod\limits_{1\leq i<j\leq \#_I(X)} (t_i-t_j)^2 \prod \limits_{i=1}^{\#_I(X)}\rho^{\mathscr S}_{I,X}(t_i),
\end{equation}
where $Z(I,X)$ is the normalization constant
and the function $\rho^{{\mathscr S}}_{I,X}$ satisfies, for any $p,q\in I$, the relation
\begin{equation}\label{rhopix-sin}
\frac{\rho^{\mathscr S}_{I,X}(p)}{\rho^{{\mathscr S}}_{I,X}(q)}=\lim\limits_{n\to\infty}  \prod\limits_{x\in X\setminus I: |x|\leq n^4}
\left(\frac{x-p}{x-q}\right)^2.
\end{equation}
\end{proposition}

In the particular case when the interval $I$ contains zero, the proposition can be formulated in a simpler way. 
For $t\in \RR$, write
\begin{equation}\label{pv-sin}
\prod^{v.p.}\limits_{x\in X\setminus I} \left( 1-\frac {t}x\right)^2=\lim\limits_{n\to\infty}  \prod\limits_{x\in X\setminus I: |x|\leq n^4}  \left( 1-\frac {t_i}x\right)^2
\end{equation}

\begin{proposition}\label{rho-sine-zero}
Let $I$ be a compact interval on ${\mathbb R}$ containing zero.
For $\Prob_{\mathscr S}$-almost any configuration $X\in \Conf({\mathbb R})$, the conditional measure $\mathbb{P}_{{\mathscr S}}(\cdot | X; {\mathbb R}\setminus I)$ has the form
\begin{equation}\label{eq-cond-sine-zero}
Z(I,X)^{-1} \prod\limits_{1\leq i<j\leq \#_I(X)} (t_i-t_j)^2 
\prod \limits_{i=1}^{\#_I(X)} \prod^{v.p.}\limits_{x\in X\setminus I} \left( 1-\frac {t_i}x\right)^2,
\end{equation}
where $Z(I,X)$ is the normalization constant.
\end{proposition}

In particular, it follows from Theorem 1.4 in \cite{buf-aop}, Theorem 1.4 and Corollary 1.5 in \cite{buf-cond} that 
the  measure $\Prob^0_{\SIN}$, the reduced Palm measure at zero for the sine-process, is quasi-invariant under the group of translations $g_s: X\to X+s, s\in\mathbb R$, and the Radon-Nikodym derivative is given by the formula 
$$
\frac{d\Prob^0_{\SIN}\circ g_s}{d\Prob^0_{\SIN}}(X)=|G_X(-s)|^2.
$$
Theorem \ref{mainthm}, together with a criterion of Kaimanovich \cite{kaim}, implies
\begin{corollary}\label{palm-cons}
The Palm measure $\Prob^0_{\SIN}$ of the sine-process is conservative under the action of the group $g_s$.
\end{corollary}
\noindent {\bf{Remark.}} It would be interesting to obtain a proof of Theorem \ref{mainthm} by directly establishing Corollary \ref{palm-cons}.

\subsection{The proof of the second claim of Theorem \ref{mainthm}}\label{int-ker}
Assume that our kernel $\Pi$ has integrable form
: there exists an open set $U\subset {\mathbb R}$ satisfying $\mu({\mathbb R}\setminus U)=0$ and  linearly independent smooth functions $A$, $B$  defined on $U$ such that
\begin{equation}
\label{piab}
\Pi(x, y) = \displaystyle \frac{A(x)B(y) - A(y)B(x)}{x-y}, x\neq y.
\end{equation}
We  assume that the functions $A,B$ never simultaneously take value $0$ on $U$. The smoothness assumption is  only needed in the continuous case. 
The term integrable comes from the connection with the theory of integrable systems discovered by Its, Izergin, Korepin and Slavnov in \cite{IIKS}.
\begin{proposition}\label{excess-two}
\label{ex-phi-1}
If the  kernel $\Pi$ satisfies 
\begin{equation}\label{pi-hs}
\displaystyle \int\limits_{{\mathbb R}} \frac{\Pi(x,x)}{1+x^2}dx<+\infty,
\end{equation}
then, for $\mathbb P_{\Pi}$-almost every $X\in\Conf(\RR)$ and any distinct  $p, q\in X$ there exists a function 
$\varphi_{p, q,X}\in L$ such that  $\varphi_{p, q,X}(y)=0$ for all $y\in X\setminus \{p,q\}$.
\end{proposition}

We start by reformulating the classification  of conditional measures for our processes in the following form. 
Let $I_k$ be an increasing exhausting sequence of compact  intervals containing $0$. 

\begin{proposition}
The conditional measure $\Prob_{\Pi}(\cdot |X, \mathbb R\setminus I_k)$ is an 
orthoognal polynomial ensemble of degree $N_k=\#_{I_k}(X)$  and has the form
$$
Z(I_k, X)^{-1} \prod\limits_{1\leq i< j\leq N_k} (t_i-t_j)^2 \prod \limits_{i=1}^{N_k} w_{X, I_k}(t_i)dt_i. 
$$
For $k>l$, the weights $w_{X, I_k}(t)$ satisfy 
\begin{equation}\label{weight-ind}
\frac{w_{X, I_k}(t)}{w_{X, I_l}(t)}=\prod\limits_{x\in X\cap (I_k\setminus I_l)} \left(1-\frac tx\right)^2.
\end{equation}
\end{proposition}

Take $q\in U$. For a configuration $X$, let $p_0=p_0(X)$ be the particle closest to $q$ (and if there are two, the smaller one).
By definition of the Radon-Nikodym derivative, for any $t\in U$ we have 
\begin{equation}\label{rn-taut}
\displaystyle \int\limits_{\Conf(\RR)}  \frac{d\Prob^t}{d\Prob^q} (X) d\Prob^q(X)=1.
\end{equation}
From \eqref{pi-hs}, \eqref{rn-taut}, 
for any $R>0$ we have 
\begin{equation}\label{comp-est}
\displaystyle \int\limits_{\{X: |p_0(X)|<R\}} \displaystyle \int\limits_{\{t: |t-q|>R+1\}}  \frac{d\Prob^t}{d\Prob^q} (X) \times  \frac{\Pi(t,t)d\mu(t)} {(t-p_0(X))^2}d\Prob^q(X)<+\infty.
\end{equation}
It follows that for $\Prob^q$-almost every $X$ we have 
\begin{equation}\label{int-rn-weak}
\displaystyle \int\limits_{\RR}  \frac{d\Prob^t}{d\Prob^q} (X)\times  \frac{\Pi(t,t) d\mu(t)}{(t-q)^2} <+\infty.
\end{equation}
By definition, the function
$$
\varphi^{(k)}_{p, q, X}(t)=\chi_{I_k}(t) \prod\limits_{x\in X\cap I_k\setminus \{p,q\}} \left(1-\frac tx\right)^2 \sqrt{w_{I_K, X}(t)}
$$
belongs to $\mathrm{Ran}  \ \Pi^{[X;\RR\setminus I_k]}$. Furthermore,  by \eqref{comp-est}, the $L_2$-norm of the functions $\varphi^{(k)}_{p, q, X}$ is uniformly bounded in $k$, while, by \eqref{weight-ind}, for $k>l$ we have the consistence property
\begin{equation}\label{phi-consis}
\varphi^{(k)}_{p, q, X}|_{I_l}=\varphi^{(l)}_{p, q, X}
\end{equation}

Relation~\eqref{comp-est}, the consistency  \eqref{phi-consis} and the Kolmogorov Compactness Criterion imply that the family of functions $\varphi^n_{p,q, X}$, $n\in\mathbb N$, is precompact in $L_2(\RR)$. Letting $\varphi_{p, q, X}^\infty$ be a limit point, as desired, we obtain $\varphi_{p, q, X}^\infty\in L$  and 
$$
\varphi_{p,q,X}^\infty(x)=0  \ \mathrm{for} \  x\in X\setminus \{p,q\}. \qed
$$ 

{\bf{Remark.}} Using Theorem 1.1 in \cite{BQ-cmp}, the argument above can also be applied to  determinantal point processes on $\mathbb C$ induced by Hilbert spaces of holomorphic functions satisfying the assumption (1) in \cite{BQ-cmp}. In particular, for the determinantal point process induced by the kernel $K(z,w)=\exp(z\overline w-|z|^2/2-|w|^2/2)$
of orthogonal projection on the Fock space, this argument shows, cf. Perelomov \cite{perelom}, that almost every configuration has excess at most $3$.

\subsection{Derivation of Theorem \ref{mainthm} from
Lemma \ref{main-lemma}}\label{der-thm-lem}
As before, for a configuration $X\in \Conf(\RR)$ and $p\in X$, write
$$
G_X^p(t)=\prod\limits_{x\in X}^{v.p.} \left(1-\frac tx\right)=\lim\limits_{n\in \mathbb N, n\to\infty}\prod_{x\in X, |x|<n^4} \left(1-\frac tx\right).
$$

Let $X\in\Conf(\mathbb R)$, let $p\in X$.  Our aim is to prove that if there exists a function $\varphi\in \PW$ such that 
$\varphi(x)=0$ for all $x\in X\setminus p$, then we must have 
$$
\prod\limits_{x\in X\setminus p}^{v.p.} \left(1-\frac tx\right)\in \PW.
$$
Once this statement is proved, Lemma \ref{main-lemma} implies Theorem \ref{mainthm}.

Let $I_k$ be an exhausting growing sequence of intervals symmetric with respect to the origin and containing $p,q$.
We now choose a specific conditional kernel $\SIN^{[X, \mathbb R \setminus I_k]}$: namely, we let  $\SIN^{[X, \mathbb R \setminus I_k]}$ be the operator of orthogonal projection onto the subspace 
$\PW^{[X, \mathbb R \setminus I_k]}$   given by the formula 
$$
\PW^{[X, \mathbb R \setminus I_k]}=\left\{P_l(x)\prod\limits_{x\in X\setminus I_n}^{v.p.} \left(1-\frac tx\right)\right\},
$$
where $P_l$ is a polynomial of degree at most $\#_{I_k}(X)$.

 Let $\varphi\in \PW$ satisfy
$$
\varphi(x)=0, x\in X\setminus p.
$$
For any $k\in \mathbb N$ we then have 
$$
\chi_{I_k}\varphi\in \PW^{[X, \mathbb R \setminus I_k]}.
$$
It follows that there exist $a,b\in \mathbb C$ such that the function $\varphi$ has the form
$$
\varphi(t)=(at+b) \prod\limits_{x\in X\setminus p}^{v.p.} \left(1-\frac tx\right).
$$
If $a=0$, then the Lemma is proved. We can thus assume $a=1$. 
If $b\notin \mathbb R$, then, recalling that the Paley-Wiener space is closed under conjugation, 
we see that 
$$
\prod\limits_{x\in X\setminus p}^{v.p.} \left(1-\frac tx\right)=\frac{\varphi-\overline \varphi}{\Im b}\in \PW.
$$
If $b\in \mathbb R$, then 
$$
\prod\limits_{x\in X\setminus p}^{v.p.} \left(1-\frac tx\right)=\frac{\varphi}{t+b}\in \PW.
$$
The lemma is proved completely. \qed

\subsection{Derivation of Lemma \ref{main-lemma}  from Lemma \ref{main-lemma-bis}}\label{der-lem-lem}

For $A\in \RR$, write $$X+iA=\{x+iA: x\in X\}.$$

 If $t_0$ is such that $S_{f^{t_0,A}}(X)>2\log t_0 $, then, exponentiating and recalling that the 
average of the exponent is greater than the exponent of the average, gives 
$$
\displaystyle \int_0^1|G_{X+iA}(t_0-u)|^2 du \geq t_0^2,
$$
and if the set of such numbers $t_0$ is infinite, then the function
 $$\displaystyle \frac{G_{X+iA}(t+iA)}{\sqrt{1+t^2}}$$ cannot be square-integrable. Consequently, 
Lemma \ref{main-lemma-bis} directly implies that, for $A$ sufficiently large, the set  
\begin{equation}\label{fX-notint}
\{X\in \Conf(\RR): \frac{G_{X+iA}(t+iA)}{\sqrt{1+t^2}} \notin L_2(\R)\}
 \end{equation}
has positive probability under the sine-process. 
For any $A\in \R$ and any  $X\in \Conf(\RR)$ such that the function $G_X$ is a well-defined entire function, the function $G_{X+iA}$ is also well-defined, and writing, for any $t\in \R$, $x\in\RR$, the bound
$$
\frac{(t-x)^2+A^2}{x^2}\geq \frac{(t-x)^2+A^2}{x^2+A^2},
$$
we obtain the inequality
$$
|G_X(t+iA)|^2\geq |G_{X+iA}(t)|^2.
$$
We thus obtain that, for sufficiently large $A$, the set
\begin{equation}\label{fX-notint-bis}
\{X\in \Conf(\RR): \frac{G_{X}(t+iA)}{\sqrt{1+t^2}} \notin L_2(\R)\}
 \end{equation}
has positive probability under the sine-process.  
We next check that the set 
\begin{equation}\label{fX-notint-tris}
\{X\in \Conf(\RR): \frac{G_{X}(t)}{\sqrt{1+t^2}} \notin L_2(\R)\}
 \end{equation}
has positive probability under the sine-process. 

Given particles $p_1, \dots, p_l\in X$, write 
$$
G_X^{\breve{p}_1, \dots, \breve{p}_l}(t)=\prod^{v.p.}_{x\in X\setminus \{p_1, \dots, p_l\}} \left(1-\frac tx\right).
$$ 
It is clear that 
$$
 \frac{G_X(t)}{\sqrt{1+t^2}} \notin L_2(\R)
 $$
 if and only if, for any particle $p\in X$, we have 
 $$
 G_X^{\breve{p}} \notin L_2(\R).
 $$
 
 Assume, to the contrary, that $
 G_X^{\breve{p}} \in L_2(\R).
 $
 Then also 
 $G_X^{\breve{p}} \in \PW.
 $
 It follows that, for any $ A\in \RR$ we have 
 $G_X^{\breve{p}}(\cdot+iA) \in \PW,
 $
 whence, in turn, we obtain 
 \begin{equation}\label{}
  \frac{G_X(t+iA)}{\sqrt{1+t^2}} \in L_2(\R),
 \end{equation}
 
 a contradiction.

It remains to show that the set \eqref{fX-notint-tris} has probability $1$.

The group of diffeomorphisms with compact support acts on $\Conf(\RR)$, and the sine-process 
is quasi-invariant under this action \cite{buf-aop}. 
The Ghosh-Peres rigidity property for the sine-process directly  implies that 
the action of our group is also ergodic: a measurable set invariant under all diffeomorphisms with 
compact support  has measure either $0$ or $1$.  

Note that the orbits of the group of diffeomorphisms with compact support are analogues, in our situation, of the leaves  the {\it symmetric  equivalence relation}  on the space of configurations
$\Conf(\RR)$: indeed, a configurations $X$  can be taken to a configuration $Y$ by a diffeomorphism with compact support if and only if 
the configurations $X$ and $Y$  coincide beyond a bounded interval and have the same number of particles within this interval. 

It is clear that the set
\begin{equation}\label{nonint-set}
\{X\in \Conf(\RR): \frac{f_{X+iA}(t)}{\sqrt{1+t^2}}\notin L_2(\RR) \}
\end{equation}
is indeed invariant under all diffeomorphisms with compact support. From \eqref{nonint-pos-dens}
 it follows that the set \eqref{nonint-set} has posiitve probability, whence, in turn, 
 it follows that the set \eqref{nonint-set} has probability  $1$. Assuming Lemma \ref{main-lemma-bis}, Lemma \ref{main-lemma}  and, consequently, Theorem \ref{mainthm} are proved completely.
\qed

It remains to prove Lemma \ref{main-lemma-bis}.

\section{Continual Toeplitz Kernels and the scaling limit of the Borodin-Okounkov-Geronimo-Case formula.}
\subsection{Sobolev semi-norms,  continual Toeplitz and Hankel operators}
For $p>0$, as usual, the symbol $H_{p}$ stands for the Sobolev $p$-space. 
The symbol $\|\varphi\|_{H_{p}}$ stands for the corresponding Sobolev {\it semi}-norm:
\begin{equation}\label{sobolev-semi}
\|\varphi\|_{H_{p}}=\displaystyle \int\limits_{\mathbb R} |\la|^{2p} |\widehat \varphi(\la)|^2d\la,
\end{equation}
and the symbol $\langle , \rangle_{H_p}$ for the Hermitian-linear form corresponding to the semi-norm 
\eqref{sobolev-semi}. 
The square of the $H_p$-norm  of a function  $\varphi\in H_p$ has the form 
$\|\varphi\|_{H_{p}}^2+\|\varphi\|_{L_{2}}^2$. While, of course, $H_{p_1}\subset H_{p_2}$ if $p_1<p_2$, 
the finiteness of the semi-norm $\|\varphi\|_{H_{p_2}}$ does not in general imply the finiteness of the semi-norm $\|\varphi\|_{H_{p_1}}$.

Introduce the Dirichlet spaces  $$
\mathscr D_2^+=\mathcal H_2^+\cap H_{1/2}, 
\mathscr D_2^-=\mathcal H_2^-\cap H_{1/2}.
$$

For a function $h$, let the function $\widetilde {h}$ be defined by $\widehat{\widetilde {h}}(t)=\widehat{h}(-t)$.

Let $f\in  L_{2}(\RR)\cap L_{\infty}(\RR)$  and introduce the continual Toeplitz operator $\mathfrak{T}(f)$  acting in $L_2(\RR)$ by the formula
$$
\mathfrak{T}(f)\psi(s)=\displaystyle \int_0^{+\infty} \widehat{f}(s-t)\psi(t)\,dt.
$$
Define also  $\mathfrak{T}(1)=1$ so that we have 
$$
\mathfrak{T}(1+f)=1+\mathfrak{T}(f).
$$
 For $h\in  L_2(\RR)\cap L_{\infty}(\RR)$  we define the continual Hankel operator 
$\mathfrak H(h)$ by the formula
$$
\mathfrak H(h)\psi(s)=\displaystyle \int_0^{\infty} \widehat h(s+t)\psi(t)dt.
$$
Define also  $\mathfrak{H}(1)=0$.

Given functions $b_1, b_2\in  L_2(\RR)\cap L_{\infty}(\RR)$, by definition we have 
\begin{equation}\label{multiplic-rel}
\mathfrak{T}((1+b_1)(1+b_2))=\mathfrak{T}(1+b_1)\mathfrak{T}(1+b_2)-\mathfrak{H}(b_1)\mathfrak{H}(\widetilde b_2).
\end{equation}
It follows that for $b\in \mathscr D_2^+$, as well as for  $b\in  {\mathscr D_2}^-$ we have 
\begin{equation}\label{exp-rel}
\exp \mathfrak{T}(b) = \mathfrak{T}(\exp b).
\end{equation}

We now take two functions $\varphi_+\in \mathscr D_2^+\cap L_{\infty}$, 
$\varphi_-\in {{\mathscr D}_2}^-\cap L_{\infty}$. 
Set
\begin{equation}\label{phi-notation}
 \varphi=\varphi_++\varphi_-, 
 g=\exp(\varphi), h=\exp(\varphi_--\varphi_+).
 \end{equation}

\begin{theorem} \label {det-sine-bo}
The operators $\mathfrak H(h)$, $\mathfrak H(\widetilde {h^{-1}})$ belong to the Hilbert-Schmidt class, 
the determinant $\det (1+(g-1)\mathscr S)$ is well-defined in the sense of \eqref{det-def}, and we  have
\begin{multline}\label{bo-id-form1}
\det (1+(g-1)\mathscr S)=\exp(\widehat \varphi(0)+\langle \varphi_+, \widetilde \varphi_-\rangle_{\onehalf})
\times\\ \times \det(1-\chi_{[2\pi, +\infty)})\mathfrak H(h)\mathfrak H(\widetilde {h^{-1}})\chi_{[2\pi, +\infty)}). 
\end{multline}
\end{theorem}
We can equivalently rewrite \eqref{bo-id-form1} in the form
\begin{multline}\label{bo-id-form2}
\det (1+(g-1)\mathscr S)=\exp(\displaystyle \frac1{2\pi}\int\limits_{\mathbb R} \varphi(t)dt+
\displaystyle \int\limits_0^{\infty} \la \widehat \varphi_+(\la)\widehat \varphi_-(-\la)d\la)\times\\ \times \det(1-\chi_{[2\pi, +\infty)})\mathfrak H(h)\mathfrak H(\widetilde {h^{-1}})\chi_{[2\pi, +\infty)}). 
\end{multline}
 
\begin{lemma}\label{BO-id} In the notation \eqref{phi-notation}  we have
\begin{equation}\label{widom-cont}
\det \left(1-\mathfrak H(h)\mathfrak H(\widetilde {h^{-1}})\right)=\exp(
-\langle \varphi_+, \widetilde \varphi_-\rangle_{\onehalf}).
\end{equation}
For any $a>0$ we have
\begin{multline}\label{BO-cont}
\det (\chi_{[0,a]} \mathfrak T(g) \chi_{[0,a]})=\exp(a\widehat \varphi(0)+
\langle \varphi_+, \widetilde \varphi_-\rangle_{\onehalf}) \det(1-\chi_{[a, +\infty)})\mathfrak H(h)\mathfrak H(\widetilde {h^{-1}})\chi_{[a, +\infty)}). 
\end{multline}
\end{lemma} 
 
 {\bf Remark.}  If $g$ is real,  then $\mathfrak H(h)^*=\mathfrak H(\widetilde {h^{-1}})$ and, since  
 $|h|=1$, also $\|\mathfrak H(h)\|\leq 1$, whence 
 $$
 \det(1-\chi_{[a, +\infty)})\mathfrak H(h)\mathfrak H(\widetilde {h^{-1}})\chi_{[a, +\infty)})\leq 1.
 $$
 
\subsection{Proof of Lemma \ref{BO-id}}
The argument goes back to Widom \cite{Widom1}, \cite{Widom2}, Basor and Widom \cite{basor}. 
The exposition follows Boettcher \cite{boet1}, \cite{boet2}  and Simon \cite{simon-opuc}.
Using \eqref{multiplic-rel}, write 
$$
\mathfrak T(h)\mathfrak T(h^{-1})=1-\mathfrak H(h)\mathfrak H(\widetilde {h^{-1}}).
$$
Again using \eqref{multiplic-rel}, write
$$
\mathfrak T(h)=\mathfrak T(\exp(\varphi_-))\mathfrak T(\exp(-\varphi_+)); \mathfrak T(h^{-1})=\mathfrak T(\exp(-\varphi_-))\mathfrak T(\exp(\varphi_+)).
$$
We thus have 
\begin{multline}
\mathfrak T(h)\mathfrak T(h^{-1})=\\=\mathfrak T(\exp(\varphi_-))\mathfrak T(\exp(-\varphi_+))\mathfrak T((\exp(-\varphi_-)))\mathfrak T((\exp(\varphi_+))=\\=\exp(\mathfrak T(\varphi_-))\exp(-\mathfrak T(\varphi_+))\exp(-\mathfrak T(\varphi_-))\exp(\mathfrak T(\varphi_+)).
\end{multline}
The Helton-Howe identity
$$
\det(\exp A \exp B \exp(-A) \exp(-B))=\exp \tr [A,B],
$$
valid for any pair $A,B$ of Hilbert-Schmidt operators, now gives
$$
\det \mathfrak T(h)\mathfrak T(h^{-1})=\exp(-\tr [\mathfrak T(\varphi_-), \mathfrak T(\varphi_+)]).
$$
Now, of course,  
$$
[\mathfrak T(\varphi_-), \mathfrak T(\varphi_+)]=\mathfrak T(\varphi_-) \mathfrak T(\varphi_+) - \mathfrak T(\varphi_+) \mathfrak T(\varphi_-),
$$
and, since $\mathfrak T(\varphi_-) \mathfrak T(\varphi_+)=\mathfrak T(\varphi_-\varphi_+)$, using 
\eqref{multiplic-rel} yet again, we obtain
$$
 [\mathfrak T(\varphi_-), \mathfrak T(\varphi_+)]=\mathfrak H(\varphi^+)\mathfrak H(\widetilde \varphi^-).
$$
By definition, we have 
$$
\tr \ \mathfrak H(\varphi^+)\mathfrak H(\widetilde \varphi^-) = \displaystyle 
\int\limits_0^{\infty} \la \widehat \varphi_+(\la)\widehat \varphi_-(-\la)d\la.
$$
We finally obtain
$$
\det(\mathfrak T(h)\mathfrak T(h^{-1}))=\exp(-\displaystyle \int\limits_0^{\infty}
 \la \widehat \varphi_+(\la)\widehat \varphi_-(-\la)d\la),
$$
and \eqref{widom-cont} is proved.
To derive the continuous analogue of the Borodin-Okounkov-Geronimo-Case identity from the Widom formula, we use the Jacobi-Dodgson identity, which in our situation takes the form
\begin{multline}\label{jac-dod}
\det(\chi_{[0,a]}(\mathfrak T(h)\mathfrak T(h^{-1})^{-1}\chi_{[0,a]})=\\=
\det(\mathfrak T(h)\mathfrak T(h^{-1})^{-1})\det(1-\chi_{(a, +\infty)}\mathfrak H(h)\mathfrak H(\widetilde {h^{-1}})\chi_{(a, +\infty)}).
\end{multline}
It remains to check, for any  $a>0$, the relation 
$$
\det(\chi_{[0,a]}(\mathfrak T(h)\mathfrak T(h^{-1}))^{-1}\chi_{[0,a]})=\exp(-a\widehat \varphi(0))
\det(\chi_{[0,a]}\mathfrak T(g)\chi_{[0,a]}).
$$
Write 
\begin{multline}
\mathfrak T(h)\mathfrak T(h^{-1}))^{-1}=
\exp(-\mathfrak T(\varphi_+))\exp(\mathfrak T(\varphi_-))
\exp(\mathfrak T(\varphi_+))\exp(-\mathfrak T(\varphi_-))=\\=\exp(-\mathfrak T(\varphi_+))
\mathfrak T(g)\exp(-\mathfrak T(\varphi_-)).
\end{multline}

We have $$\chi_{[0,a]}\mathfrak T(\exp(-\varphi_+))\chi_{[a, +\infty)}=
\chi_{[a, +\infty)}\mathfrak T(\exp(-\varphi_-))\chi_{[0,a]}=0,$$
and $$\det \chi_{[0,a]}\mathfrak T(\exp(\varphi_-))\chi_{[0,a]}=\exp(a\widehat \varphi_-(0)),$$
$$ 
\det \chi_{[0,a]}\mathfrak T(\exp(\varphi_+))\chi_{[0,a]}=\exp(a\widehat \varphi_+(0)).$$  We thus obtain
\begin{multline}
\det(\chi_{[0,a]}(\mathfrak T(h)\mathfrak T(h^{-1}))^{-1}\chi_{[0,a]})=\\=
\det(\chi_{[0,a]}\mathfrak T(\exp(-\varphi_+))\chi_{[0,a]}\mathfrak T(f) \chi_{[0,a]}\mathfrak T((\exp(-\varphi_+)\chi_{[0,a]})))=\\=\exp(-a\widehat \varphi_+(0)-a\widehat \varphi_-(0))\det(\chi_{[0,a]}\mathfrak T(g) \chi_{[0,a]}).
\end{multline}
The  proof is complete.

The second Lemma gives a connection between Toeplitz determinants and the generalized
Fredholm determinants, in the sense of \eqref{det-def}, related to the sine-kernel. 

\begin{lemma}\label{toeplitz-sine} Let $g$ be a bounded continuous function on $\RR$,  such that the function 
$g-1$ vanishes at inifnity, is square-integrable,  and the  integral
$$
\int\limits_{\RR} (g(t)-1)dt
$$
exists in principal value.
Then we have
\begin{equation}
\mathrm{det}(\chi_{[0,2\pi]} \mathfrak T(g) \chi_{[0,2\pi]})=\mathrm{det}(1+(g-1)\mathscr S).
\end{equation}
\end{lemma}

\subsection{Proof of Lemma \ref{toeplitz-sine}} 
Let $$K_n(\theta, \theta^{\prime})=\frac{\sin(\frac{n+1}{2}(\theta-\theta^{\prime}))}{\sin(\frac{\theta-\theta^{\prime}}{2})}$$
be the standard Dirichlet kernel.
We will need the following approximation to the sine-kernel.
\begin{equation}\label{approx-kernel}
\breve K_n(t,t^{\prime})=
\chi_{[-\pi n,\pi n]}(t)\chi_{[-\pi n,\pi n]}(t^{\prime})\frac1{n}\frac{\sin(\pi \frac{(n+1)(t-t^{\prime})}{n})}
{\sin(\pi\frac{t-t^{\prime}}{n})}. 
\end{equation}
As $n\to\infty$, we have 
\begin{equation}\label{kn-sin}
\breve K_n(t,t^{\prime})\to \SIN(t, t')
\end{equation}
uniformly on compact sets.
Note the relation between the approximating kernel  $\breve K_n$ and the Dirichlet kernel $K_n$:
\begin{equation}\label{approx-dir}
\breve K_n(t,t^{\prime})=
\chi_{[-\pi n,\pi n]}(t)\chi_{[-\pi n,\pi n]}(t^{\prime})\frac1{n}K_n(2\pi t/n, 2\pi t^{\prime}/n). 
\end{equation}
Now let  $f$ be such that $f-1$ is a Schwartz function on $\RR$.
Set 
$$
{f}_n(t)=\chi_{[-\pi n,\pi n]}\sum\limits_{k\in\ZZ} \hat f(k/n) \exp(2\pi itk/n).
$$
Then we have  $$(f_n-1)\breve K_n\to (f-1)\mathscr S$$ uniformly on compact sets and in the 
Hilbert-Schmidt class. Furthermore, since $K_n(t,t)=\mathscr S(t,t)=1$, we have 
$$
\displaystyle \int\limits_{\RR} (f_n(t)-1) K_n(t,t)dt\to \displaystyle \int\limits_{\RR} (f(t)-1)\mathscr S(t,t) dt.
$$
It follows that 
$$
\det(1+((f-1)\mathscr S)=\lim\limits_{n\to\infty} \det(1+(f_n-1)\breve K_n).
$$
Next, set 
$$
\mathring{f}_n(t)=\chi_{[-\pi,\pi]}\sum\limits_{k\in\ZZ} \hat f(k/n) \exp(2\pi itk). 
$$ 
Making a change of variable $\theta=2\pi t/n$ and recalling \eqref{approx-dir}, we have the identity
$$
\det(1+(\mathring{f}_n-1)K_n)=\det(1+(f_n-1)\breve K_n).
$$
It follows that 
$$
\det(1+((f-1)\mathscr S)=\lim\limits_{n\to\infty} \det(1+(f_n-1)K_n).
$$

Next, we note that the function $\mathring{f}_n$ can be considered as a function on the circle, and we have 
$$
\det (T_n(\mathring{f}_n))=\det(1+(\mathring{f}_n-1)K_n),
$$ 
since, indeed, for any function $g\in L_2(\TT)$, we  have 
$$
\det (T_n(g))=\displaystyle \int \limits_{\TT^n} \prod\limits_{j=1}^n g(\theta_j) \prod\limits_{j<k}\left|\exp(i\theta_j)-\exp(i\theta_k)\right|^2 \prod\limits_{j=1}^n \frac{d\theta_j}{2\pi}=
\det(1+(g-1)K_n).
$$
Finally, we verify that 
\begin{equation}\label{det-t-conv}
\det(\chi_{[0,2\pi]} \mathfrak T(f) \chi_{[0,2\pi]})=\lim\limits_{n\to\infty}\det (T_n(\mathring{f}_n)).
\end{equation}
For $n\in \mathbb N$, let $L_2^{(n)}\subset L_2$ be the subspace of functions constant on intervals $[2\pi\frac kn, 2\pi\frac{k+1}n]$.
We can naturally identify $L_2^{(n)}$ with $l_2(\ZZ)$, and let ${\breve T}_n$ be the operator with matrix $T(\mathring{f}_n)$ 
acting in $L_2^{(n)}$ extended as an operator on the whole $L_2$ by acting as the zero operator on the orthogonal complement
to $L_2^{(n)}$ in $L_2$. Note that the operator  $\chi_{[0,2\pi]}{\breve T}_n\chi_{[0,2\pi]}$ acts with the matrix 
$T_n(\mathring{f}_n)$. If  $f-1$ is a Schwartz function, then we have 
$$
\chi_{[0,2\pi]}{\breve T}_n\chi_{[0,2\pi]}-1\to \chi_{[0,2\pi]} \mathfrak T(f) \chi_{[0,2\pi]}-1
$$
in the Hilbert-Schmidt class, and also that 
$$
\tr(\chi_{[0,2\pi]}{\breve T}_n\chi_{[0,2\pi]}-1)\to \tr(\chi_{[0,2\pi]} \mathfrak T(f) \chi_{[0,2\pi]}-1),
$$
whence \eqref{det-t-conv}. 

We next approximate a general function $g$ by Schwartz functions $g_n$ in such a way that 
$g_n\to g$ in $L_2(\RR)$, $\max|g_n-g|\to 0$ and 
$$
\int\limits_{\RR} g_n-g \to 0;
$$
the integral above is understood in principal value.

We then have $$\det(1+(g-1)\mathscr S)=\lim\limits_{n\to\infty}\det(1+(g_n-1)\mathscr S);
$$
$$\det(\chi_{[0, 2\pi]} \mathfrak T(g) \chi_{[0,2\pi]})=\lim\limits_{n\to\infty}\det(\chi_{[0, 2\pi]} \mathfrak T(g_n) \chi_{[0,2\pi]}).$$ 
 Lemma \ref{toeplitz-sine} is proved.

\subsection{The Hilbert-Schmidt norm of restricted Hankel operators}
Our next aim is to give estimates for the determinant
\begin{equation}\label{det-hh}
 \det(1-\chi_{[2\pi, +\infty)})\mathfrak H(h)\mathfrak H(\widetilde {h^{-1}})\chi_{[2\pi, +\infty)}).
\end{equation}
\begin{lemma}
Let $h=\exp(f)$. For any $d>0$  we then have 
\begin{multline}\label{hh-est}
|\det\left(1-\chi_{[d, +\infty)})\mathfrak H(h)\mathfrak H(\widetilde {h^{-1}})\chi_{[d, +\infty)}\right)-1|\leq \\ \leq 
 \frac{ \|f\|^2_{H_1}\exp (2\|f\|_{L_{\infty}})}{d} \exp\left(\frac{ \|f\|^2_{H_1}\exp (2\|f\|_{L_{\infty}})}{d}\right).
\end{multline}
\end{lemma}
\begin{proof}
Letting the symbol $\|A\|_{\mathscr I_1}$  stand for the trace-class norm of the operator $A$, we start with the clear inequality 
$$
|\det(1-A)- 1|\leq \|A\|_{\mathscr I_1}\exp (\|A\|_{\mathscr I_1}).
$$
Using the notation $\|\cdot\|_{HS}$ for the Hilbert-Schmidt norm, write 
$$
 \|\chi_{[d, +\infty)})\mathfrak H(h)\mathfrak H(\widetilde {h^{-1}})\chi_{[d, +\infty)}\|_{\mathscr I_1}\leq 
 \|\chi_{[d, +\infty)})\mathfrak H(h)\|_{HS} \|\chi_{[d, +\infty)})\mathfrak H(\widetilde {h^{-1}})\|_{HS}.
$$
We naturally have 
$$
\|\chi_{(d, +\infty)} \mathfrak H(h)\|^2_{HS}=\displaystyle \int\limits_d^{\infty} |u| \widehat h(u)|^2 du\leq 
\frac{\|h\|^2_{H_1}}{d}.
$$
We arrive at the estimate
$$
 \|\chi_{[d, +\infty)})\mathfrak H(h)\mathfrak H(\widetilde {h^{-1}})\chi_{[d, +\infty)}\|_{\mathscr I_1}\leq 
 \frac{\|h\|_{H_1}\|\widetilde {h^{-1}}\|_{H_1}}{d}.
$$

Next, for a bounded function $f\in H_1$ by definition we have  
\begin{equation}\label{h1-est}
\|\exp(f)\|_{H_1}\leq  \|f\|_{H_1}\exp (\|f\|_{L_{\infty}}),
\end{equation}
whence 
$$
 \|\chi_{[d, +\infty)})\mathfrak H(h)\mathfrak H(\widetilde {h^{-1}})\chi_{[d, +\infty)}\|_{\mathscr I_1}\leq 
 \frac{ \|f\|^2_{H_1}\exp (2\|f\|_{L_{\infty}})}{d},
$$
and \eqref{hh-est} is established.\end{proof}

\subsection{An estimate for the gradient of the Hankel determinant and the conclusion of the proof of Lemma \ref{indep-lemma-lowfreq}}
Our next step is an estimate for the gradient of the determinant of the form \eqref{det-hh} considered as a function of the parameter.
Let $f_1,\dots, f_n\in H_1(\RR)\cap L_{\infty}(\RR)$, let $h(\vec z)=h(z_1, \dots, z_n)=\exp(z_1f_1+\dots+z_nf_n)$, 
and set 
\begin{equation}
c(z_1, \dots, z_n)=\det(1-\chi_{[d, +\infty)})\mathfrak H(h(\vec z))\mathfrak H(\widetilde {h^{-1}(\vec z)})\chi_{[d, +\infty)})
\end{equation}
\begin{lemma}\label{value-grad}
For any $M>0, n>0$ there exists a constant $C=C(n, M)$, $C(n,M)>0$,  such that for any $z_1, \dots, z_n$ satisfying 
$|z_1|, \dots, |z_n|\leq M$ we have 
\begin{multline}
| \|\mathrm{grad}  \ c(z_1, \dots, z_n)\|-1|\leq Cn\max \limits_{l=1,\dots, n} \|f_l\|_{H_1}\exp(
Cn\max \limits_{l=1,\dots, n} \|f_l\|_{L_{\infty}})\times \\ \times\exp(Cn\max \limits_{l=1,\dots, n} \|f_l\|_{H_1}\exp(
Cn\max \limits_{l=1,\dots, n} \|f_l\|_{L_{\infty}})).
\end{multline}

\end{lemma}
\begin{proof}The lemma is a direct corollary of \eqref{hh-est} and the Cauchy integral formula for the derivative.
\end{proof}

\subsection{Conclusion of the proof of Lemma \ref{indep-lemma-lowfreq}}

 We derive Lemma \ref{indep-lemma-lowfreq}  from Theorem \ref{det-sine-bo}.
 As before, we consider $m$ fixed, while $A$ is assumed to take values in a bounded interval 
 $(a_0, A_0)$, where $a_0, A_0$ depend only on $m$ and satisfy $1<a_0<A_0$.
 First, note that, by definition, we have 
 $$
 \|F^{T,A}\|^2_{\onehalf}=2\log T.
 $$
 Next, a direct computation using the formula \eqref{fourier-f}   gives 
 \begin{proposition}\label{norm-fdal}
 For any $m\in \mathbb N$  there  exists  a positive constant $C$ such that 
for any $T\in \mathbb N$, any $d=T, T+1, \dots, 2T$, any  $l=1, \dots, m-1$ we have 
 $$
 \left|\|f^{d,A, l}\|^2_{\onehalf}-2\log T/m\right|\leq C; \ \langle  f_+^{d,A,l},  F_+^{T}\rangle_{\onehalf} \leq C.
 $$
 \end{proposition}
 
\begin{proposition}\label{cov-fdal}
For any $m\in \mathbb N$  there  exists  a positive constant $C$ such that 
for any $T\in \mathbb N$, any  $l=1, \dots, m-1$, any natural $d_1, d_2\in [T,  2T]$ satisfying $d_1-d_2>T^{l/m}$, 
we have
\begin{equation}
\langle  f_+^{d_1,A,l},  f_+^{d_2,A,l}\rangle_{\onehalf} \leq C.
\end{equation}
\end{proposition}

We now apply Theorem \ref{det-sine-bo} to the additive statistic $S_{f^{d_1, d_2, A, T}(\la, \vec\la^{(1)},\vec\la^{(2)})}$.
We have 
$$
\int_{\mathbb R} f^{d_1, d_2, A, T}(\la, \vec\la^{(1)},\vec\la^{(2)})(t)dt=0.
$$
By Propositions  \ref{norm-fdal}, \ref{cov-fdal}, there exists a constant $C$ depending only on $m$ such that 
\begin{equation}\label{norm-fdalla}
\left|\|f^{d_1, d_2, A, T}(\la, \vec\la^{(1)},\vec\la^{(2)})\|^2_{\onehalf} - 
\la^2\log T+\frac{\log  T}{m}(\sum\limits_{r=1}^{m-1} (\la^{(1)}_r)^2+\sum\limits_{s=1}^{l}(\la^{(2)}_s)^2\right|<C.
\end{equation}

Our next aim is to check that the Hilbert transforms $Hf^{d, A, r}$ are uniformly bounded. It suffices to 
establish the uniform boundedness of the Hilbert transforms $H\phi^{d, A, r}$.

\begin{proposition}\label{hilbtr-fdal}
For any $m\in \mathbb N$  there  exists  a positive constant $C$ such that 
for any $T\in \mathbb N$, any  $l=1, \dots, m-1$, any natural $d\in [T,  2T]$ 
we have
\begin{equation}\label{est-hilb-bdd}
\|Hf^{d, A, l}\|_{L_{\infty}}<C; \|H\phi^{d, A, l}\|_{L_{\infty}}<C.
\end{equation}
\end{proposition}
\begin{proof}
We start by establishing a uniform bound for  the Hilbert transform $H\phi^{d, A}$.  That for $H\phi^{d, A, l}$,
$Hf^{d, A, l}$ follows.
From \eqref{fourier-phi} we obtain 
\begin{equation}
\label{fourier-hilbert-phi}
\widehat{H\phi^{d, A}}(\la)=\frac{\exp(i\la d)-1)}{\la}\exp(-A|\la|).
\end{equation}
For $d\in [T, 2T]$, we can represent the function $\widehat{H\phi^{d, A}}(\la)$ as a sum of three terms:
$$
\widehat{H\phi^{d, A}}(\la)=\psi_1^{d, A}+\psi_2^{d, A}+\psi_3^{d, A},
$$
where 
$$
\psi_1^{d, A}(\la)=-\la^{-1}\mathbb I_{(-\infty, -T^{-1}]\cup [T^{-1}, \infty)};
$$
$$
\psi_2^{d, A}(\la)=\frac{\exp(i\la d)}{\la}\exp(-A\la)\mathbb I_{(-\infty, -T^{-1}]\cup [T^{-1}, \infty)};
$$
$$
\psi_3^{d, A}(\la)=\widehat{H\phi^{d, A}}(\la)\mathbb I_{(-T^{-1}, T^{-1})}.
$$
The uniform boundedness, in $T$ and in $A$,  of  the $L_{\infty}$-norm of the inverse Fourier transforms of $\psi_1$ and $\psi_2$ is checked directly; it is here that we see the difference between $\phi^{d, A}$ and $H\phi^{d, A}$: if we write $|\la^{-1}|$ instead of $\la^{-1}$, then the corresponding $L_{\infty}$-norms grow in $T$.  Next, note that 
$\psi_3$ has the form $\psi_3(\la)=\psi_{33}(T\la)\mathbb I_{(-T^{-1}, T^{-1})}(\la)$, where $\psi_{33}$ is a fixed smooth function; the boundedness of  the $L_{\infty}$-norm of the inverse  Fourier transforms of $\psi_3$ follows.

The desired uniform bound for  the Hilbert transform $H\phi^{d, A}$ is obtained.  That for $H\phi^{d, A, l}$,
$Hf^{d, A, l}$ follows.
\end{proof}

As usual, decompose
$$
f^{d_1, d_2, A, T}(\la, \vec\la^{(1)},\vec\la^{(2)})=f^{d_1, d_2, A, T}(\la, \vec\la^{(1)},\vec\la^{(2)})_++
f^{d_1, d_2,A, T}(\la, \vec\la^{(1)},\vec\la^{(2)})_-, 
$$
$f^{d_1, d_2, A, T}(\la, \vec\la^{(1)},\vec\la^{(2)})_+\in \mathcal H_2^+$, 
$f^{d_1, d_2, A, T}(\la, \vec\la^{(1)},\vec\la^{(2)})_-\in \mathcal H_2^-$, 
and set
$$
h^{d_1, d_2, A, T}(\la, \vec\la^{(1)},\vec\la^{(2)})=\exp(f^{d_1, d_2, A, T}(\la, \vec\la^{(1)},\vec\la^{(2)})_-- f^{d_1, d_2, A, T}(\la, \vec\la^{(1)},\vec\la^{(2)})_+).
$$
Set 
$$
B(l,m)=\{(\theta, \vec \theta^{(1)}, \vec \theta^{(2)}):\theta\in \mathbb R, \theta^{(1)}\in \mathbb R^{m-1}, \la^{(2)}\in \mathbb R^l: |\theta|+ |\vec \theta^{(1)}|+|\vec  \theta^{(2)}|<R\}.
$$
Directly from the definitions we have 
\begin{proposition}\label{norm-hdal}
For any $\varepsilon>0$, any $m\in \mathbb N$, there exists a constant $a_0$ depending only on $m$ and 
$\varepsilon$ such that the following holds for any $A>a_0$. For any $R>0$, all sufficiently large $T\in \mathbb N$, any  $l=1, \dots, m-1$, 
any natural $d_1, d_2\in [T,  2T]$ we have 
$$
 \max\limits_{(\theta, \vec \theta^{(1)}, \vec \theta^{(2)})\in B(l,m,R)}\| \chi_{[2\pi, +\infty)}\mathfrak H (h^{d_1, d_2, A, T}(\la, \vec\theta^{(1)},\vec\theta^{(2)}))\|_{HS} <\varepsilon.
$$
\end{proposition}
\begin{proof} A direct computation gives, uniformly in $l$ and in $R$, the relation 
$$
\lim\limits_{A\to\infty} \max\limits_{(\theta, \vec \theta^{(1)}, \vec \theta^{(2)})\in B(l,m,R)}\| f^{d_1, d_2, A, T}(\theta, \vec\theta^{(1)},\vec\theta^{(2)})\|_{H_1} = 0.
$$
The proposition follows now from the uniform Hilbert transform estimate \eqref{est-hilb-bdd} together with the 
estimate \eqref{h1-est}.
\end{proof}
Theorem \ref{det-sine-bo}, the estimate of Lemma \ref{value-grad},  Propositions \ref{norm-fdal}, 
\ref{cov-fdal}, \ref{norm-hdal}, together  imply Lemma \ref{indep-lemma-lowfreq}. \qed

\subsection{An estimate on the  speed of convergence in the Soshnikov Central Limit 
Theorem for additive statistics}
Let $T_af(u)=f(u/a)$. We  have 
$$
\|\chi_{(d, +\infty)} \mathfrak H(T_af)\|_{HS}=\|\chi_{(ad, +\infty)} \mathfrak H(f)\|_{HS}.
$$
It follows  that $\|\chi_{(d, +\infty)} \mathfrak H(T_af)\|_{HS}\to 0$ as $a\to\infty$, whence, for any $\la\in \mathbb C$ we obtain  
$$
\expsin \exp (\la S_{T_af}) \to \exp (\la^2 \|f\|^2_{\onehalf}/2). 
$$ 
We arrive at a different proof of the theorem of Soshnikov \cite{soshnikov} that the random variables 
$$
\frac{S_{T_af}}{\|f\|_{\onehalf}}
$$
converge in law to the standard Gaussian as $a\to\infty$.  For the Circular Unitary Ensemble and the Gaussian Unitary Ensemble  convergence in law, without normalization, of additive statistics to the Gaussian distribution, is due to Johansson \cite{johansson-annals-1997}, \cite{johansson-duke}.
We now estimate the speed of convergence.  
\begin{lemma}\label{kolmogorov} Let $f_+\in H_1\cap \mathcal H_2 \cap L_{\infty}$ be a bounded smooth function on $\RR$ and let $f=f_++\overline f_+$.  There  exists a positive constant $c$  depending only on $\|f\|_{H_1}$, $\|f_+\|_{L_{\infty}}$, such that the Kolmogorov-Smirnov distance between the distribution of the random variable
$$
\frac{S_{T_af}}{\|f\|_{\onehalf}}
$$
and the standard Gaussian is at most 
$
\displaystyle \frac {c}{\log\la}.
$
\end{lemma}
\begin{proof}
The bound \eqref{hh-est} implies that for any real-valued function $v\in L_{\infty}\cap H_1$ 
there exists a constant $c>0$ depending only on $\|v\|_{L_{\infty}}$, $\|v\|_{H_1}$ such that we have 
\begin{multline}\label{double-exp-est}
|\det(1-\chi_{[2\pi, +\infty)})\mathfrak H(T_a\exp(\la v))\mathfrak H(\widetilde {T_a\exp(-\la v)})\chi_{[2\pi, +\infty)})-1|\leq 
\\ \leq 
\frac {c\la\exp (c\la)}{a}\exp\left( c\la \frac {\exp (c\la)}{a}\right).
\end{multline}
Now normalize $f$ in such a way that $\|f\|_{\onehalf}=1$, let $v=Hf$ be the Hilbert transform of $f$,  
and set $$\psi_a(\la)=\expsin \exp (i\la S_{T_af}).$$ 
It follows from \eqref{double-exp-est} that there exist constants $C_1$, $C_2$ such that we have
$$
\int\limits_{-C_1\log\la}^{C_1\log\la} \frac{|\psi_a(\la)-\exp(-\la^2/2)|}{\la}\leq \frac {C_2}{\log\la}.
$$
Lemma \ref{kolmogorov} directly follows now from the Feller smoothing estimate on the Kolmogorov-Smirnov distance.
\end{proof}

\section{Multiplicative Functionals depending on a pair of particles and the change of variable formula.}

\subsection{An outline of the section}
The  proof of Lemma \ref{indep-lemma-highfreq}, the main exponential estimate on the decay of the 
characteristic function at  high frequencies of an additive functional of the sine-process, proceeds by a  change of variables, following a method first used by Johansson \cite{johansson-annals-1997}, \cite{johansson-duke}. The Jacobian  of our change of variables involves a multiplicative functional depending on a pair of particles. Our first aim in this section is therefore to obtain estimates 
on expectations, under the sine-process, of multiplicative functional depending on a pair of particles. 
 We start with  a  super-exponential upper estimate for the probability of a large number of particles in a fixed interval. 
 This estimate, which holds in considerable generality, also implies desired estimates on the expectation of our multiplicative functionals corresponding to observables exceeding one. We next need to consider more general 	observables as well as {\it regularized} multiplicative functionals. The scheme is similar to that used in \cite{buf-aop} for multiplicative functionals depending on one particle: we first obtain estimates on the variance  
 of additive functionals  depending on a pair of particles, then introduce regularized additive functionals by continuity,
 and, finally, pass on to multiplicative functionals.

\subsection{Tails of the number of particles in an interval}
Let $\Pi$ be an Hermitian kernel, smooth in the totality of variables and inducing an orthogonal projection. 
 As before, let $\Prob_{\Pi}$ be the corresponding determinantal point process. Write 
$$\Pi^{(k)}(x,y)=\frac{\partial^k}{\partial y^k}\Pi^{(k)}(x,y).$$
Assume that there exists a positive constant $C_0=C_0(\Pi)$ such that 
\begin{equation}\label{sl-gr-der}
\sup\limits_{x,y\in \mathbb R, |x-y|\leq 1}| \Pi^{(k)}(x,y)|\leq C_0^k k!
\end{equation}
The sine-kernel clearly satisfies \eqref{sl-gr-der}.
\begin{proposition}
If \eqref{sl-gr-der} holds, then, for any interval $I$ there exists a constant $\alpha>0$ depending only on the length of $I$ and such that 
\begin{equation}\label{k-sq}
\Prob_{\Pi}(\#_I\geq k) \leq \exp(-\alpha k^2).
\end{equation}
\end{proposition}
The point process $\Prob_{\Pi}$ has negative correlations, and it suffices to prove \eqref{k-sq} for a sufficiently short interval $I$: the general case follows by partitioning. It therefore suffices to establish the following
\begin{lemma} \label{ksq-small-lem} If the kernel $\Pi$ satisfies \eqref{sl-gr-der}, then, 
 for any interval $I$ of length less than $(1+2C_0)^{-1}$ we have 
 \begin{equation}\label{ksq-small}
 \Prob (\#_I\geq k)\leq ((1+2C_0)|I|)^{k(k+1)/2}.
\end{equation}
\end{lemma}

\begin{proof}
The inequality  \eqref{ksq-small} is a direct corollary of the estimate 
\begin{equation}\label{det-dd-est}
\det \left(\Pi(x_i, x_j)\right)_{i,j=1, \dots, k}\leq 
k!((1+2C_0)|I| )^{k(k-1)/2}.
\end{equation}
We now establish \eqref{det-dd-est}.
 Fix $x_1, \dots, x_k$. Introduce the divided differences 
 $$
 \Pi[x; x_1, x_2, \dots; x_l]
 $$
 inductively by writing 
 $$
 \Pi[x; x_1, x_2]=\frac{\Pi(x,x_1)-\Pi(x,x_2)}{x_1-x_2}
 $$
 and 
 $$
 \Pi[x; x_1, x_2, \dots, x_{l+1}]=\frac{\Pi[x; x_1, x_2, \dots, x_l]-\Pi[x; x_2, \dots, x_{l+1}]}{x_1-x_{l+1}}.
 $$
By definition of the determinant, we have 
$$
\det(\Pi(x_i, x_j))_{i,j=1, \dots, k}=\Delta(x_1, \dots, x_k)\det(Q_{il}),
$$
where $$\Delta(x_1, \dots, x_k)=\prod\limits_{1\leq i<j\leq k} (x_i-x_j)$$ is the Vandermonde 
determinant of $x_1, \dots, x_k$ and we set
$Q_{i1}=\Pi(x_i, x_1)$, while for $l\geq 2$ we set
$$
Q_{il}({\vec x})=\Pi[x_i;  x_1, x_2, \dots, x_l].
$$ 
For any $x\in I$ there exists $y\in I$ such that 
$$
\Pi[x ,x_1, \dots, x_l]=\frac1{l!}\Pi^{(l)}(x,y).
$$
The estimate \eqref{det-dd-est} follows, and  Lemma \ref{ksq-small-lem} is proved.
\end{proof}

The estimate \eqref{k-sq} implies that the random variable $\#_I^2$ admits exponential moments of sufficiently small order. To make  this statement  more precise, we establish 
\begin{lemma} \label{exp-square}
Let $\alpha>0$.
 If $Z$ is a random variable taking non-negative integer values 
and satisfying,
for any $k$, the inequality
$$
\Prob(Z=k)\leq \exp(-\alpha k^2),
$$
then, for any $\gamma<\alpha$, we have 
\begin{equation} \label{expxi-est}
\ee \exp(\gamma Z^2)\leq \exp\left(\frac{\gamma}{(1-\exp(\gamma-\alpha))(1-\exp(-\alpha))}\right).
\end{equation}
\end{lemma}

\begin{proof}
We have \begin{multline}
\ee \exp(\gamma Z^2)=\Prob(Z=0)+\sum_{k=1}^{\infty} \Prob(Z=k)\exp(\gamma k^2)\leq \\ \leq 1+
\sum_{k=1}^{\infty} 
(\exp((\gamma-\alpha) k)-\exp(-\alpha k))\leq\\ \leq  \exp\left(\frac{\gamma}{(1-\exp(\gamma-\alpha))(1-\exp(-\alpha))}\right).
\end{multline}
\end{proof}

\subsection{Discrete norms}
 For $k\in\ZZ$, let $I(k)=[k, k+1]$.
Let $q=q(x,y)$ be a function of $2$ real variables. 
Let $\BB(1, 1)$  be the space of bounded functions of $2$ variables endowed with the norm
$$
\|q\|_{\BB(1, 1)}=\sum\limits_{k,l\in\ZZ}\sup\limits_{x\in I(k),y\in I(l)}|q(x,y)|.
$$

Let $\BB(1, \infty)$  be the space of bounded functions of $2$ variables endowed with the norm
$$
\|q\|_{\BB(1, \infty)}=\sup\limits_{k\in\ZZ}\sum\limits_{k,l\in\ZZ}\sum\limits_{x\in I(k),y\in I(l)}|q(x,y)|.
$$

For a continuous function $f$ on $\RR$ write 
$$
\|f\|^2_{\BB(1)}=\sum\limits_{k\in\ZZ}  \max\limits_{x\in I_k} |f(x)|^2;
$$
$$
\|f\|^2_{\BH(1/2)}=\| \left(\frac{f(x)-f(y)}{x-y}\right)^2\|_{\BB(1,1)}.
$$
$$
\|f\|^2_{\BH(1)}=\| \left(\frac{f(x)-f(y)}{x-y}\right)^2\|_{\BB(1,\infty)}.
$$
We let $\BH(1/2)$ and $\BH(1)$ be spaces of functions, considered modulo additive constants, obtained as the completions of the spaces of compactly 
supported smooth functions with respect to the 
norms $\|\cdot\|_{\BH(1/2)}$, $\|\cdot\|_{\BH(1)}$, respectively.

\begin{lemma}\label{h-half-est}
\begin{enumerate}
\item
For any $f\in C^1(\RR)\cap \onehalf \cap \BH(1/2)$ we  have 
 $$
\|f\|^2_{\BH(1/2)} \leq 3\|f\|^2_{\onehalf}+11\|f'\|^2_{\BB(1)}  .
 $$
 \item For any $f\in C^1(\RR)\cap \BH(1)$ satisfying $f'\in L_{\infty}$ we  have 
 $$
 \|f\|^2_{\BH(1)}\leq   11(\|f'\|^2_{L_{\infty}}+\|f\|^2_{L_{\infty}}).
  $$
 \end{enumerate}

\end{lemma}

\begin{proof}
First, we clearly have 
$$
\sum\limits_{k\in\ZZ} \max\limits_{x, y: x,y\in I_k} \left|\frac{f(x)-f(y)}{x-y}\right|^2\leq \|f'\|^2_{\BB(1)};
$$
$$
 \sum\limits_{k\in\ZZ}\max\limits_{x, y: x\in I_k, y\in I_{k+1}} \left|\frac{f(x)-f(y)}{x-y}\right|^2\leq 2\|f'\|^2_{\BB(1)},
$$
and, indeed, more generally, for any natural $r$ one can write 
$$
\sum\limits_{k\in\ZZ}\max\limits_{x, y: x\in I_k, y\in I_{k+r}} \left|\frac{f(x)-f(y)}{x-y}\right|^2\leq r\|f'\|^2_{\BB(1)}.
$$
Write $b_k(f)= \max\limits_{x: x\in I_k} |f(x)|^2$.
 For any $k,l\in\ZZ$ and any points $x, x', y, y'$ satisfying  $x, x'\in I_k,   y, y'\in I_l,$ we have 
 $$
 |f(x)-f(y)|^2\ge \frac{|f(x')-f(y')|^2}3-(b_k(f')+b_l(f')).
 $$
whence 
$$
\sum\limits_{k\in\ZZ, l\in\ZZ: |k-l|\geq 2}\max\limits_{x, y: x\in I_k, y\in I_{k+r}} \left|\frac{f(x)-f(y)}{x-y}\right|^2\leq 
3\|f\|_{\onehalf}^2 + 8\|f'\|^2_{\BB(1)},
$$
and the proof of the first claim is complete. The second claim is proved in the same way.
\end{proof}

\subsection{Positive multiplicative functionals over pairs of particles}

Now let $\Prob$ be a point process on $\RR$ with negative correlations and admitting a 
positive constant $\alpha_{\Prob}$ such that for any 
interval $I$ whose endpoints are consecutive integers and any natural $k$ we have
\begin{equation}\label{def-c}
\Prob(\#_I=k)\leq \exp(-\alpha_{\Prob} k^2).
\end{equation}
For a determinantal process $\Prob_K$ induced by the kernel $K$, for brevity we shall write 
$\alpha_K=\alpha_{\Prob_K}$; in particular, we write $\alpha_{\SIN}$ for the sine-process.
Lemma \ref{exp-square} directly implies 
\begin{corollary}
For any subinterval $I\subset \RR$ there exists a constant $c>0$  depending only on 
the length of $I$ such that for all $\gamma\in (0, c)$ we have 
$$
\ee_{\Prob_{\Pi}} \exp(\gamma \#_I^2)\leq \exp\left(\frac{\gamma}{(1-\exp(\gamma-\alpha_{\Prob}))(1-\exp(-\alpha_{\Prob}))}\right)
$$
\end{corollary}
Let $q: E^2\to \mathbb C$ be a symmetric Borel function: $q(x,y)=q(y,x)$  satisfying $q(x,x)=0$.
Introduce the additive functional
$$
S_q(X)=\sum\limits_{\{x, y\}\subset X} q(x,y).
$$ 
\begin{lemma}\label{lem-bi-mult} For any $q\in \BB(1, 1)$ and any $\gamma<\alpha_{\Prob}/\|q\|_{\BB(1, \infty)}$ we have
\begin{equation}\label{upper-bi-mult}
\ee \exp(\gamma S_{ q}) \leq  \exp\left(\frac{\gamma\|q\|_{\BB(1, 1)}}
{(1-\exp(\gamma\|q\|_{\BB(1, \infty)}-\alpha_{\Prob}))(1-\exp(-\alpha_{\Prob}))}\right).
\end{equation}
\end{lemma}
\begin{proof}
For brevity, set $\#_k=\#_{[k, k+1]}$. 
For $k,l\in\ZZ$, set 
$$
\qmax(k,l)=\sup\limits_{x\in [k, k+1] ,y\in [l, l+1]}|q(x,y)|
$$
$$
\qmax(k)=\sum\limits_{l\in\ZZ}\qmax(k,l).
$$
Write 
$$
S_{ q}\leq \sum\limits_{k,l\in\ZZ} \qmax(k,l)\#_k\#_l\leq \sum\limits_{k\in\ZZ} \qmax(k)\#_k^2.
$$
Our point process $\Prob$  has negative correlations, whence 
$$
\ee_{\Prob}\exp(\gamma S_{ q})\leq \prod\limits_{k\in\ZZ} \exp(\gamma\qmax(k)\#_k^2).
$$
Now, for each $k$, apply \eqref{expxi-est} to the random variable $Z=\#_k^2$, and the lemma follows.
\end{proof}

\begin{corollary}\label{expsq-lp}
For any $q\in \BB(1, 1)$ and any $\gamma\in (0, \alpha/\|q\|_{\BB(1, \infty)}$, there exists $p>1$ depending only on $\gamma$ such that we have 
$$
\exp(\gamma S_{ q})\in L_p(\Conf(E), \Prob).
$$
\end{corollary}

We next turn to establishing the continuity properties of the multiplicative functional $\exp(\gamma S_{ q})$ as a function of $q\in \BB(1, 1)$.
We start with the following clear 
\begin{proposition}\label{nbhd-zero-q}
For any $p>0$ there exists a neighbourhood $U$ of $0$ in $\BB(1, 1)$ such that the correspondence 
$q\to \exp( S_{ q})$ yields a continuous map from $U$ to $L_p(\Conf(E), \Prob)$.
\end{proposition}
\begin{proof}
First note that for any  $\varepsilon>0$, in a sufficiently small neighbourhood 
of $0$ we have 
\begin{equation}\label{ee-ineq}
|\ee_{\Prob} \exp( S_{ q})-1|<\varepsilon.
\end{equation}
 Indeed, the desired upper bound on $\ee_{\Prob} \exp( S_{ q})$ directly follows from \eqref{upper-bi-mult}, while the lower bound follows from the upper and the Cauchy-Bunyakovsky-Schwarz inequality by writing
$$
1 \leq \ee_{\Prob} \exp( S_{ q})\ee_{\Prob} \exp( S_{-q}).
$$
For any  $n\in\NN$ and $\varepsilon>0$, expanding the product and using \eqref{ee-ineq}, in a sufficiently small neighbourhood 
of $0$ we have $\ee_{\Prob} |\exp( S_{ q})-1|^{2n}<\varepsilon$. \end{proof}

\begin{corollary}
For any $q\in \BB(1, 1)$ and any $p\in (0, \alpha/\|q\|_{\BB(1, \infty)}$, there exists a neighbourhood $U$ of $q$ 
in $\BB(1, 1)$ such that the correspondence 
$q\to \exp( S_{ q})$ yields a continuous map from $U$ to $L_p(\Conf(E), \Prob)$.
\end{corollary}
\begin{proof} Given $q'\in \BB(1)$, write  $\exp( S_{ q'})=\exp( S_{ q})\exp( S_{q'- q})$, apply Lemma 
\ref{nbhd-zero-q} to $q'-q$ and use the H\"older inequality.
\end{proof}

\subsection{Regularized additive functionals over pairs of particles}

As before, $E$ is a locally compact complete metric space and $\mu$ a sigma-finite Borel measure on $E$.
Fix an exhausting sequence of bounded sets in $E$, assume that for any $x\in E$ the integral 
$$
 \int_{E} q(x,y)d\mu(y)
 $$
 exists in principal value:
 $$
 \int_{E}^{v.p.} q(x,y)K(y,y)d\mu(y)=\lim\limits_{n\to\infty}  \int_{B_n} q(x,y)K(y,y)d\mu(y)
 $$
  and write 
 $$
 Q(x)=\int_{\RR}^{v.p.} q(x,y)K(y,y)d\mu(y). 
$$
\begin{proposition}\label{var-bi}
Let $c_0>0$. Let $K$ be a reproducing kernel of an orthogonal projection acting in $L_2(E, \mu)$ 
and satisfying
$$
\sup\limits_{x,y\in E} |K(x,y)|<c_0.
$$
There exists a constant $C$ depending only on $c_0$ such that 
we have 
$$
\mathrm{Var}_{\Prob_K} S_q\leq C(\|Q\|^2_{L_2(E, \mu)}+\|q\|^2_{L_2(E^2, \mu^2)}).
$$
\end{proposition}
\begin{proof}  The reproducing property gives, for $\mu$-almost every $x\in E$, the identity 
\begin{equation}\label{repro-prop}
K(x,x)=\int\limits_E |K(x,y)|^2 d\mu(y).
\end{equation}
In this proof the  symbols $C_{kl}$ will stand for positive constants depending only on $c_0$. 
Let $\sigma$ be a permutation of $3$ elements. We prove
\begin{equation}\label{three-perm}
\int\limits_{E^3} q(x_1, x_2)q(x_1, x_3) \prod\limits_{j=1}^3 K(x_j, x_{\sigma j})dx_1dx_2dx_3\leq  
C_3(\|Q\|^2_{L_2(E, \mu)}+\|q\|^2_{L_2(E^2, \mu^2)}).
\end{equation}
For the identity permutation, \eqref{three-perm} is clear by definition of $Q$. Let $\sigma \neq \mathrm{id}$.
If $\sigma(2)=2$, then $\sigma(3)=1$. We integrate the function $q(x_1, x_2)$ in $x_2$ and apply the Cauchy-Bunyakovsky-Schwarz inequality, first in $x_3$, and then in $x_1$, to the functions $Q(x_1)$ and the function $q(x_1, x_3)K(x_3, x_1)$. The case $\sigma(3)=3$ is similar. If $\sigma (2)=3$ and $\sigma (3) =2$, then the desired bound \eqref{three-perm} follows from the reproducing property \eqref{repro-prop} of the kernel $K$ and the Cauchy-Bunyakovsky-Schwarz inequality in $x_1, x_2, x_3$ applied to the functions $q(x_1, x_2)K(x_2, x_3)$ and   $q(x_1, x_3)K(x_2, x_3)$.
If the permutation $\sigma$ is a cycle, then  the desired bound \eqref{three-perm} follows from the reproducing property \eqref{repro-prop} of the kernel $K$ and the Cauchy-Bunyakovsky-Schwarz inequality in $x_1, x_2, x_3$ applied to the functions $q(x_1, x_2)K(x_3, x_{\sigma 3})$ and   $q(x_1, x_3)K(x_2, x_{\sigma 2}) K(x_1, x_{\sigma 1})$.

Let $\sigma$ be a permutation of $4$ elements satisfying $\sigma(\{1, 2\})\neq \{1, 2\}$. We prove the bound 
\begin{equation}\label{four-perm}
\int\limits_{E^3} q(x_1, x_2)q(x_3, x_4) \prod\limits_{j=1}^4 K(x_j, x_{\sigma j})dx_1dx_2dx_3dx_4\leq  
C_4(\|Q\|^2_{L_2(E, \mu)}+\|q\|^2_{L_2(E^2, \mu^2)}).
\end{equation}

Without losing generality, assume $\sigma (1)=3$. If  $\sigma (3)=1$, $\sigma (2)=2$ and $\sigma (4)=4$, then we integrate in the variables $x_2$ and $x_4$, after which the desired bound \eqref{four-perm} follows from the reproducing property \eqref{repro-prop} of the kernel $K$ and the Cauchy-Bunyakovsky-Schwarz inequality in $x_1, x_3$ applied to the functions $Q(x_1)K(x_1, x_3)$ and   $Q(x_3)K(x_3, x_1)$.

If $\sigma (3)=1$, $\sigma (2)=4$ and $\sigma (4)=2$,  then the desired bound \eqref{four-perm} follows from the reproducing property \eqref{repro-prop} of the kernel $K$ and the Cauchy-Bunyakovsky-Schwarz inequality in $x_1, x_2, x_3, x_4$ applied to the functions $q(x_1, x_2)K(x_1, x_3)K(x_2, x_4)$ and   $q(x_3, x_4)K(x_3, x_1)K(x_4, x_2)$.

If $\sigma (3)=2$ and $\sigma (4)=4$, then we integrate in the  variable $x_4$, after which the desired bound \eqref{four-perm} follows from the reproducing property \eqref{repro-prop} of the kernel $K$ and the Cauchy-Bunyakovsky-Schwarz inequality in $x_1, x_2, x_3$ applied to the functions $Q(x_3)K(x_1, x_3)K(x_2, x_1)$ and   $q(x_1, x_2)K(x_3, x_2)$. The case $\sigma (3)=4$ and $\sigma (2)=2$ is  identical by a permutation of the symbols.

Finally, we consider the case in which our permutation $\sigma$ is a four-cycle. If $\sigma (3)=2$, then 
the desired bound \eqref{four-perm} follows from the reproducing property \eqref{repro-prop} of the kernel $K$ and the Cauchy-Bunyakovsky-Schwarz inequality in $x_1, x_2, x_3, x_4$ applied to the functions $q(x_1, x_2)K(x_1, x_3)K(x_2, x_4)$ and   $q(x_3, x_4)K(x_3, x_2)K(x_4, x_1)$. If $\sigma(3)=4$, then 
the desired bound \eqref{four-perm} follows from the reproducing property \eqref{repro-prop} of the kernel $K$ and the Cauchy-Bunyakovsky-Schwarz inequality in $x_1, x_2, x_3, x_4$ applied to the functions $q(x_1, x_2)K(x_1, x_3)K(x_3, x_4)$ and   $q(x_3, x_4)K(x_4, x_2)K(x_2, x_1)$.
\end{proof}

Let $\QQ$ be the subspace of  functions $q\in L_2(E^2, \mu^2)$ on $E^2$ such that the integral
$$
 i_q(x)=\int_{\RR}^{v.p.} q(x,y)dy
$$
is defined for $\mu$-almost all $x\in E$  and satisfies $i_q\in L_2(E, \mu)$.
The space $\QQ$ is turned into a Hilbert space by setting
$$
\|q\|_{\QQ}^2=\|q\|_{L_2(E^2, \mu^2)}^2+\|i_q\|_{L_2(E, \mu)}^2.
$$
The correspondence 
$$
q\to \overline S_q=S_q-\ee S_q
$$
 is first defined for compactly supported bounded functions $q$ and then extended by continuity to the whole of 
 $\QQ$. By Proposition \ref{var-bi}, we obtain a continuous linear mapping defined on $\QQ$ and taking values in 
 $L_2(\Conf(E), \Prob_K)$. The regularized additive functional $\overline S_q$ may be defined even when the usual 
  additive functional $ S_q$ is not. If,  additionally, the integral  
 \begin{equation}\label{def-rho2}
 \rho_2^K(q)=\int\limits_{E\times E}^{v.p.} q(x,y) (K(x,x)K(y,y)-|K(x,y)|^2) d\mu(x)d\mu(y)
 \end{equation}
 converges, 
 then we can set $S_q=\overline S_q+\rho_2^K(q)$.
 Note that if  $q_1, q_2$ satisfy $q_1-q_2\in L_2(E^2, \mu^2)$ and the integral $\rho_2^K(q_1)$ converges, then so does $\rho_2^K(q_2)$.
 \subsection{Regularized multiplicative functionals over pairs of particles}
 We now turn from additive to multiplicative functionals. First, let $g$ be a symmetric 
 Borel function on $E\times E$ 
 such that $g(x,x)=1$ and the function $g-1$ has compact support.
 Introduce a multiplicative functional $\psitwo_g$ on $\Conf(\RR)$ by the formula
 \begin{equation}\label{def-psitwo}
 \psitwo_g(X)=\prod\limits_{\{x,y\}\subset X} g(x,y).
 \end{equation}
 The product in the right-hand side converges since it has finitely many terms. 
 We now extend the definition of the multiplicative functional $\psitwo_g$ to a larger class of functions $g$.
First, observe that   if $g$  satisfies $g\geq 1$, then the  multiplicative functional is well-defined and bounded above by $1$, 
as the product in the right-hand side of \eqref{def-psitwo} converges, possibly to zero. 
 
 Now let $\QQ_0$ be the closed subset of those functions $q\in \QQ$ that 
 satisfy  $|q|\leq 1$ and the integral \eqref{def-rho2} is well-defined. 
 For $q\in \QQ_0$ set 
 \begin{equation}\label{def-psitwo-reg}
 \psitwo_{1+q}(X)=\exp(\rho_2^K(q)+ \overline S_q(X)) \cdot \psitwo_{(1+q)\exp(-q)}.
 \end{equation}
 Note that the multiplicative functional $\psitwo_{(1+q)\exp(-q)}$ is well-defined since $(1+q)\exp(-q)\leq 1$.

\subsection{Multiplicative functionals over pairs of particles corresponding to 
complex-valued functions}
In the change  of variable formula below, we shall also need to consider multiplicative functionals corresponding to 
complex-valued functions. 
\begin{lemma}
Let $q\in\QQ$ satisfy $|q|<1$. Then the function 
\begin{equation}\label{def-q1}
q_1=\frac{(-i)\log (1+iq)}{\sqrt{1+q^2}}
\end{equation}
also belongs to $\QQ$.
\end{lemma}
\begin{proof}
By definition, we have 
$$
|q_1-q|<|q|^3.
$$
Recalling that $|q|<1$ and that $q\in L_2(E^2, \mu^2)$, we obtain that $q_1\in\QQ$.
\end{proof}

It follows that the normalized additive functional $\overline S_{q_1}$ is well-defined. 
Furthermore, if the integral $\rho_2^K(q)$, given by  \eqref{def-rho2}, is well-defined, then 
$\rho_2^K(q_1)$ is well-defined as well.

Let $q \in\QQ_0$ satisfy $q^2\in \BB(1,1)$.  
Define $q_1$ by the formula \eqref{def-q1}.
Take $\gamma\in (0, \sqrt{\alpha_{K}/\|q^2\|_{\BB(1,\infty)}})$ and set
\begin{equation}\label{psitwo-compl} 
\psitwo_{1+i\gamma q}=\psitwo_{\sqrt{1+\gamma^2 q^2}}\exp(\overline S_{q_1}+\rho_2^K(q_1)). 
\end{equation}

\begin{corollary}\label{psitwo-q-compl-lp}
For any $q \in\QQ$ satisfying $q^2\in \BB(1,1)$  and such that $\rho_2^K(q)$ is well-defined and any
 $\gamma\in (0, \sqrt{\alpha_{K}/\|q^2\|_{\BB(1,\infty)}})$ there exists $p>1$ depending only on $\gamma$ and such that 
$$\psitwo_{1+i\gamma q}\in L_p(\Conf(E), \Prob).
$$
\end{corollary}

We  also directly obtain the continuity  of our multiplicative functional.
Introduce  a norm on the space $$\QB_0=\{q\in QQ_0: q^2\in \BB(1,1) \ \mathrm{and} \  \rho_2^K(q) \  
\mathrm{is \  well-defined} \}$$ 
by  setting 
$$
\|q\|^2_{\QB_0}= \|q\|^2_{\QQ}+ \|q^2\|_{\BB(1,1)}.
$$ 
\begin{corollary}\label{psitwo-q-cont-lp}
For any $q_0 \in\QQ$ satisfying $q^2\in \BB(1,1)$  and such that $\rho_2^K(q)$ is well-defined and any
 $\gamma\in (0, \sqrt{\alpha_{K}/\|q^2\|_{\BB(1,\infty)}})$ there exists a  neighbourhood
  $U$ of $q$ in $\QB_0$ as  well as a number $p>1$ depending only on $\gamma$ and such that the correspondence 
$$
q\to \psitwo_{1+i\gamma q}
$$
induces a continuous correspondence from $\QB_0$ to $L_p(\Conf(E), \Prob)$.
\end{corollary}

 We now turn to the particular case when  $K$ is the sine-kernel, while $q$ takes the form 
 $$
 q(x,y)=\omega[2](x,y)=\frac{\omega(x)-\omega(y)}{x-y},
 $$
 where $\omega$ is a bounded smooth function belonging to $\BH(1/2)$. 
 First, one directly checks the relations 
 $$
 \int\limits_{\RR}\int\limits_{\RR} \omega[2](x,y)dxdy=0;
 $$
 $$
 \int\limits_{\RR}\int\limits_{\RR} \omega[2](x,y) \SIN^2(x,y)dxdy=0;
 $$
 whence 
 $$
 \rho_2^{\SIN}(\omega[2])=0.
 $$
Recalling Lemma \ref{h-half-est}, we arrive at
\begin{corollary}\label{psitwo-omega-cont-lp}
For any $\omega\in H(1/2)\cap \BB(1)$  and any
 $\gamma_0\in (0, \sqrt{\alpha_{\SIN}/\|\omega[2]^2\|_{\BB(1)}})$ there exists a  neighbourhood
  $U$ of $\omega$ in $H(1/2)\cap \BB(1)$ as  well as a number $p>1$ depending only on $\gamma_0$ such that the correspondence 
$$
(\omega, \gamma)\to \psitwo_{1+i\gamma \omega[2]}
$$
induces a continuous map from $U\times (-\gamma_0, \gamma_0)$ to $L_p(\Conf(E), \Prob)$.
\end{corollary}

\subsection{Quasi-symmetries of determinantal point processes.}

Let $\omega$ be a  bounded smooth real-valued function on $\RR$. The map $F_{\omega}: x\to x+\omega(x)$ is a diffeomorphism of $\RR$ if and only if there exists $c\in (0,1)$ such that $\sup\limits_{\RR}|\omega'(x)|<1$.

Let $\GG$ be the subgroup of diffeomorphisms $F_{\omega}$ of $\RR$ additionally satisfying the requirements 
$$
\lim\limits_{|x|\to\infty} \omega(x)=0, 
\|\omega\| _{H_{1/2}}<\infty, \|\omega\| _{H_{1}}<\infty. 
$$
We norm the group $\GG$ by setting
$$
\|F_{\omega}\|_{\GG}=\|\omega\| _{L_{\infty}}+\|\omega'\| _{L_{\infty}}+
\|(F_{\omega}^{-1})'\| _{L_{\infty}}+
\|\omega\| _{H_{1/2}}+\|\omega\| _{H_{1}}.
$$
For $F\in\GG$, $F=F_{\omega}$ introduce a real-valued function $\Xi(F_{\omega})$ by setting 
\begin{equation}\label{def-xi}
\Xi(F_{\omega};X)=\prod^{v.p.}\limits_{\{x,y\}\subset X} \left(1+ \frac{\omega(x)-\omega(y)}{x-y}\right)^2  \prod^{v.p.}\limits_{x\in X} (1+\omega'(x)).
\end{equation}
We write $Xi(F_{\omega};X)$ instead of $Xi(F_{\omega})(X)$.
In the right-hand side of \eqref{def-xi},  the first multiple is a regularized multiplicative functional over pairs of particles corresponding to the function 
$$
\left(1+ \frac{\omega(x)-\omega(y)}{x-y}\right)^2,
$$
well-defined since $\|\omega\| _{H_{1/2}}<\infty$, 
whereas the second multiple is the regularized multiplicative functional over particles corresponding to the function 
$F_{\omega}'$, well-defined since $\|\omega\| _{H_{1}}<\infty$.

\begin{theorem}\label{quasi-inv}
The sine-process is quasi-invariant under the action of the group $\GG$, and for any $F\in\GG$, 
 we have the $\probsin$-almost sure identity
$$
\frac{d\probsin\circ F}{d\probsin}=\Xi(F).
$$
\end{theorem}
\begin{proof}
If $\omega$ has compact support, then the statement of Theorem \ref{quasi-inv} is proved in Theorem 1.4 in \cite{buf-aop}. We now approximate a general diffeomorphism $F\in \GG$ by those of the form $F_{\omega}$, where 
$\omega$ has compact support.
 We prepare
\begin{lemma}\label{lem-count-conv}
Let $G$ be a metrizable topological group, and let $G_0$ be a dense subgroup.  Let $\mathbf{T}=\{T_g\}$ 
be an action, continuous in the totality of the variables,  of the group $G$  on a complete separable metric  space $Y$.
Let $R_g(y)$, $g\in G$, $y\in Y$, be a Borel cocycle over the action $\mathbf{T}$.  Let $\Prob$ be a Borel probability measure on $Y$ such that 
\begin{enumerate}
\item the equality 
\begin{equation}\label{rn-id}
\frac{d\Prob\circ T_g}{d\Prob} =R_g
\end{equation}
holds $\Prob$-almost surely for any $g\in G_0$
\item for any $g\in G$ there exists a sequence $g_n\in G_0$ such that $g_n\to g$ in $G$ and $R_{g_n}\to R_g$ almost surely with respect to the measure $\Prob$.
\end{enumerate}
Then \eqref{rn-id} holds for all $g\in G$.
\end{lemma}
\begin{proof}
Take $g\in G$ and let $g_n\in G_0$ be the approximating sequence given by our second assumption.
Since the action $\mathbf{T}$ is continuous, we have $$\Prob\circ T_{g_n}\to \Prob\circ T_{g}$$ with respect to the weak topology. 
Note that, by Fatou's lemma, we have $\int\limits_{Y} R_g d\Prob\leq 1$.
Now take an arbitrary  $\varepsilon>0$. There exists a compact subset  $K\subset Y$ and a number $M>0$ 
 such that 
$R_g<M$ on $K$, $\int\limits_{Y\setminus K} R_g d\Prob<\varepsilon$, $R_{g_n}\to R_g$ uniformly on $K$,
$\Prob(K)>1-\varepsilon$. For $n$ large enough we have $|R_{g_n}- R_g|<\varepsilon$ on $K$. Since 
$\int\limits_{Y} R_{g_n} d\Prob=1$, we have $\int\limits_{Y\setminus K} R_{g_n} d\Prob<3\varepsilon$.
Now let $f:Y\to (0,1)$ be a bounded continuous function . For sufficiently large $n$ we have 
$$
\int\limits_{K}f |R_{g_n}-R_{g}| d\Prob< 4\varepsilon;
$$
whence finally 
$$
|\int\limits_{Y}f R_{g_n}d\Prob-\int\limits_{Y}R_{g} d\Prob|< 8\varepsilon;
$$
consequently, we have $R_{g_n}\Prob \to R_{g}\Prob$ weakly as $n\to\infty$, whence  \eqref{rn-id} holds for $g$ as well.
\end{proof}

We now derive Theorem \ref{quasi-inv} from Lemma \ref{lem-count-conv}. Observe first that the group 
if $F_n\in \GG$, $F\in\GG$, and $F_n\to F$ in $\GG$, then, by definition, the uniform 
convergence $F_n\to F$ takes place on compact sets, whence it follows that the group $\GG$
  acts continuously on $\Conf(\RR)$. Next, for an element  
$F\in\GG$, $F=F_{\omega}$, we construct a sequence $\omega_n$ of functions of compact support such that 
\begin{enumerate}
\item $F_{\omega_n}\in \GG$,  $F_{\omega_n}\to F_{\omega}$ in $\GG$ 
\item  $\Xi(F_{\omega_n})\to \Xi(F_{\omega})$ almost surely with respect to $\probsin$
\end{enumerate}
by setting  $\omega_n(t)= \omega(t)$ for $|t|\leq n$, $\omega_n(t)=0$ for $|t|\geq n+1$ and interpolating smoothly in $[-n-1, -n] \cup [n, n+1]$.    Since $\lim\limits_{|x|\to\infty} \omega(x)=0$, 
we see that 
\begin{equation}\label{conv-h1}
\|\omega_n - \omega\|_{H_1}\to 0,
\end{equation} 
\begin{equation}\label{conv-h2}
\|\omega_n - \omega\|_{\onehalf}\to 0
\end{equation} 
 as $n\to\infty$, whence
$F_{\omega_n}\to F_{\omega}$ in $\GG$.  We must next check the $\probsin$-almost sure convergence  
$\Xi(F_{\omega_n})\to \Xi(F_{\omega})$, perhaps  after passing to a subsequence. 
But indeed, the relation \eqref{conv-h1} implies, perhaps  after passing to a subsequence, the almost sure convergence
$$ 
\prod^{v.p.}\limits_{x\in X} (1+\omega_n'(x))\to \prod^{v.p.}\limits_{x\in X} (1+\omega'(x)),
$$
 while  the relation \eqref{conv-h2} implies, perhaps  after passing to a subsequence, the almost sure convergence
 $$
\prod^{v.p.}\limits_{\{x,y\}\subset X} \left(1+ \frac{\omega_n(x)-\omega_n(y)}{x-y}\right)^2 \to \prod^{v.p.}\limits_{\{x,y\}\subset X} \left(1+ \frac{\omega(x)-\omega(y)}{x-y}\right)^2.  
 $$  
 Theorem \ref{quasi-inv} follows now  from Lemma \ref{lem-count-conv}.
\end{proof}
{\bf {Remark.}} A similar quasi-invariance theorem  also holds for more general determinantal point processes with integrable kernels. It will be exposed in the sequel to this paper. 
\subsection{A Change of variable formula}
\begin{lemma}\label{chg-var-lem}
Let $g$ be a function holomorphic in a horizontal strip $H$ such that $g-1\in H_1(\RR)$, 
the integral 
$$
\displaystyle \int_{\RR} (g(x)-1)dx
$$
exists in principal value, and, additionally, we have 
$$
\displaystyle \int_{\RR} \sup\limits_{\{y: x+iy\in H\}} |g'(x+iy)|^2dx<+\infty
$$
Let $\omega\in \BB(1)\cap H_1(\RR)$ be a bounded smooth function satisfying 
$$\sup\limits_{x\in \RR}|\omega'(x)|<+\infty.$$ 
There exists 
$\varepsilon_0>0$ depending only on  the width of the strip $H$ and on the norms
\begin{equation}\label{omega-norms}
\|\omega\|_{L_{\infty}}, \|\omega'\|_{L_{\infty}}, \|\omega\|_{H_{1}}, \|\omega\|_{\BB(1)} 
\end{equation}
such that for all $\varepsilon$ satisfying $|\varepsilon|<\varepsilon_0$ we have
\begin{multline}\label{chg-var}
\displaystyle \int \limits_{\Conf(\RR)}\Psi_g d\mathbb{P}_{{\mathscr S}}(X)
=\displaystyle \int\limits_{\Conf(\RR)} \Psi_g(X) 
\psitwo_{1+i\varepsilon\frac{\omega(x)-\omega(y)}{x-y}}(X)\Psi_{1+i\varepsilon\omega'}(X)d\mathbb{P}_{{\mathscr S}}(X).
\end{multline}
\end{lemma}

\begin{proof}
Let $I$ be a bounded interval in $\RR$, assume that  $0\in I$. As above, we write
$$
\prod^{v.p.}\limits_{x\in X\setminus I}
\left(1-\frac{t}{x}\right)^2=\lim\limits_{n\to\infty}  \prod\limits_{x\in X\setminus I: |x|\leq n^4}
\left(1-\frac{t}{x}\right)^2
$$
and recall that the conditional measure $\mathbb{P}_{{\mathscr S}}(\cdot | X_0; C)$ 
takes the form
\begin{equation}\label {cond-sine-2}
Z(I,X)^{-1} \prod\limits_{1\leq l<j\leq \#_I(X_0)} (t_i-t_j)^2 \prod \limits_{l=1}^{\#_I(X)} \prod^{v.p}\limits_{x\in X\setminus I}
\left(1-\frac{t_i}{x}\right)^2.
\end{equation}
Let $\omega$ be a function satisfying the assumptions of the lemma and supported in $I$. Set $F_{\omega}(x)=x+i\omega(x)$. The map  $F_{\omega}$ is a diffeomorphism onto its image. Assume that $F_{\omega}(\RR)\subset H$. We take $X_0\in \Conf(\RR)$ and apply the Cauchy formula to the integral
$$
\displaystyle \int\limits_{\Conf(\RR)} \Psi_g d\mathbb{P}_{{\mathscr S}}(X | X_0; C).
$$
For brevity, set $N=\#_I(X_0)$.
The Cauchy formula gives
\begin{multline}\label{cauchy-act}
\displaystyle \int_{I^N} \prod_{l=1}^N g(t_l) \cdot 
\prod_{1\le l<j \le N} (t_l-t_j)^2 \cdot  \prod_{l=1}^N \prod_{x\in X\cap(\R\backslash I)}^{{ v.p.}} \left(1-\frac{t_l}{x}\right)^2 = \\ =
\displaystyle \int_I ... \displaystyle \int_I \prod_{i=1}^N g(F(t_i))\cdot 
\prod_{1\le i<j \le N} (F(t_i) - F(t_j))^2 \cdot \\ \cdot
\prod_{i=1}^N \prod_{x\in X\cap(\R\backslash I)}^{{ v.p.}} \left(1-\frac{F(t_i)}{x}\right)^2 F'(t_i) dt_i.
\end{multline}
Recall our notation
$$
\omega	[2](x,y)=\frac{\omega(x)-\omega(y)}{x-y}
$$
and write 
$$ 
\Psi[2;\omega]=\Psi_{\omega[2]}^{(2)}.
$$
We now rewrite \eqref{cauchy-act} in the form
\begin{multline}\label{cauchy-pre-id}
 \displaystyle \int_I ... \displaystyle \int_I \prod_{l=1}^N g(F(t_i)) \cdot
\Psi[2;\omega] \cdot \Psi_{F'} \cdot \prod_{1\le l<j\le N} (t_i-t_j)^2 \cdot \prod_{x\in X\cap(\R\backslash I)}^{{ v.p.}} \left(1-\frac{t_l}{x}\right)^2 dt_l =\\= \displaystyle \int\limits_{\Conf(\RR)} \Psi_g(X) \Psi[2;\omega](X)\Psi_{F'}(X)d\mathbb{P}_{{\mathscr S}}(X | X_0; C),
\end{multline}
thus arriving at the identity 
\begin{multline}\label{cauchy-id}
\displaystyle \int \limits_{\Conf(\RR)}\Psi_g d\mathbb{P}_{{\mathscr S}}(X | X_0; C)= \\
=\displaystyle \int\limits_{\Conf(\RR)} \Psi_g(X) \Psi[2;F](X)\Psi_{F'}(X)d\mathbb{P}_{{\mathscr S}}(X | X_0; C).
\end{multline}

In order to use the Fubini Theorem and to conclude the equality \eqref{chg-var},
we  would need to check that  the function $\Psi_g(X) \Psi[2;{\omega}](X)\Psi_{1+\omega'}(X)$ belongs to $L_1(\Conf(\RR), \probsin)$; this is where we need to pass from $\omega$ to $\varepsilon \omega$ for sufficiently small $\varepsilon$.

We start by observing that $\Psi_g$ and $\Psi_{F'_{\varepsilon\omega}}$ belong to $L_{p}(\Conf(\RR), \probsin)$ for any $p>0$. Furthermore, there exists $p_0>1, \varepsilon_0>0$  depending only on  the norms \eqref{omega-norms} such that 
$$\Psi[2;F_{\varepsilon}](X)\in L_{p_0}(\Conf(\RR), \probsin).$$

Lemma \ref{chg-var-lem} for compactly supported functions is established.
 We now pass to general functions $\omega$.

Taking $\omega_0\in \BB(1/2)\cap H_1(\RR)$ satisfying the assumptions of Lemma \ref{chg-var-lem}
we approximate our function $\omega$ by compactly supported functions $\omega_n$, $n=1,2, \dots$,  in such a way that 
$\omega_n\to \omega_0$ both in  $\BB(1)$ and in $H_1(\RR)$
as $n\to\infty$.  From convergence in $\BB(1)$, in particular, we have 
$\|\omega_n-\omega_0\|_{L_{\infty}}\to 0$ as $n\to\infty$.
 Choose $\varepsilon_0$ small enough in such a way that $x+i\varepsilon\omega_0(x)\in H$ for all $x\in \RR$, and also  $x+i\varepsilon\omega_n(x)\in H$ for all $x\in \RR$ and all sufficiently large $n$. From the relation $$
\displaystyle \int_{\RR} \max\limits_{y: x+iy\in H} |g(x+iy)|^2dx<+\infty$$ 
it follows that 
$$
\displaystyle \int_{\RR} |g(x+i\varepsilon\omega_0(x))-g(x+i\varepsilon\omega_n(x))|dx\to 0 
$$
as $n\to\infty$. 
In view of Corollaries \ref{psif-cont-sine}, \ref{psitwo-omega-cont-lp}  and the H\"older inequality, we conclude that 
$$
\Psi_{g\circ(\id+i\omega_n)}(X) \Psi[2;\omega_n](X)\Psi_{1+i\omega_n'}(X) \to \Psi_{g\circ(\id+i\omega)}(X) \Psi[2;\omega](X)\Psi_{1+i\omega'}(X)
 $$
in $L_1(\Conf(\RR), \probsin)$ as $n\to\infty$. The lemma is proved completely.
\end{proof}

\subsection{Exponential moments of  additive functionals.}

In this section the change of variables is used in order to replace rapidly oscillating additive functionals by exponentially decaying ones and to give estimates for the corresponding exponential moments. 
As before, for a function $v\in L_2(\RR)$, we let  $Hv$ stand for its Hilbert transform: 
$
\widehat {Hv}(\la)=\mathrm{sgn}(\la) \widehat v(\la).
$
\begin{lemma}
Let $u,v$ be  bounded holomorphic functions  in a horizontal strip $H$ such that 
$u,v\in H_1(\RR)$, $Hv\in \BB(1)$,
the integrals
$$
\displaystyle \int_{\RR} u(x)dx, \displaystyle \int_{\RR} v(x)dx
$$
exist in principal value. Assume,  additionally, that we have 
$$
\displaystyle \int_{\RR} \max\limits_{\{y: x+iy\in H\}} |u'(x+iy)|^2+|v'(x+iy)|^2dx<+\infty;
$$
and that there exists a constant $C_{00}$ such that 
\begin{multline}\label{fact-est}
\|u^{(k)}\|_{L_{\infty}} < C_{00}k!; \|u^{(k)}\|_{L_{1}} < C_{00}k!; \\ 
\|v^{(k)}\|_{L_{\infty}} < C_{00}k!; \|v^{(k)}\|_{L_{1}} < C_{00}k!.
\end{multline}
 Then there exist 
$p_0>1, \gamma>0, \delta>0, D>0$ depending only on  
\begin{equation}\label{constants}
C_{00},  \|Hv\|_{\BB(1)}, \|Hv'\|_{L_{\infty}}, \|v\|_{H_{1}}
\end{equation}
such that we have
\begin{multline}\label{chg-var-ineq}
\left|\displaystyle \int \limits_{\Conf(\RR)}\exp (S_{au+i\la v} d\mathbb{P}_{{\mathscr S}}(X)\right|\leq \\ \leq
\exp\left(-\gamma|\la|\|v\|^2_{\onehalf}+a\widehat u(0)+D(a^2\|u\|^2_{\onehalf} +
\|v\|^2_{\onehalf}+\|Hv\|^2_{\BB(1)}+a^2+\la^2+1)\right). 
\end{multline}
\end{lemma}
\begin{proof}
We introduce the change of  variable by the formula $F_{\varepsilon}(x)= x+i\varepsilon Hv(x)$ and use the change of variable formula \eqref{chg-var}.  For definiteness, we assume $\la>0$ and $\varepsilon>0$; in the case of negative $\la$, one should also take $\varepsilon$ to be negative.
\begin{lemma}\label{linf-hf-norm-der}
 If $\varepsilon$ is such that
 \begin{equation}\label{linf-hf}
|\eps| \|Hv\|_{L_{\infty}} < \frac{1}{2}; \
|\eps| \|Hv'\|_{L_{\infty}} < \frac{1}{2},
\end{equation}
then we have
\begin{multline}\label{changed-moment}
\displaystyle \int \limits_{\Conf(\RR)}|\exp (S_{au\circ F_{\varepsilon}+i\la v\circ F_{\varepsilon} } |d\mathbb{P}_{{\mathscr S}}(X)\leq  \\
\leq \exp(C_1+ a\displaystyle \int\limits_{\RR} u(x)dx+a^2\|u\|^2_{\onehalf}-|\varepsilon\la|\|v\|^2_{\onehalf})
\end{multline}
\end{lemma}
\begin{proof} We use Lemma \ref{BO-id} and the remark following it.
Write
\begin{multline}
\displaystyle \int \limits_{\Conf(\RR)}|\exp (S_{au\circ F_{\varepsilon}+i\la v\circ F_{\varepsilon} } |d\mathbb{P}_{{\mathscr S}}(X)=\\=\displaystyle \int \limits_{\Conf(\RR)}\exp ( S_{\Re( au\circ F_{\varepsilon}+i\la v\circ F_{\varepsilon}) } |d\mathbb{P}_{{\mathscr S}}(X)\leq  \\ \leq \exp(\displaystyle \int \limits_{\RR}\Re {(au\circ F_{\varepsilon}(x)dx+i\la v\circ F_{\varepsilon}(x)dx )}+ \|\Re {(au\circ F_{\varepsilon}+i\la v\circ F_{\varepsilon} )}\|^2_{\onehalf}) .
\end{multline}
Write the Taylor expansion for $u$:
	$$
		u(x+iHv\varepsilon (x)) = u(x) + i\varepsilon u'(x)Hv(x) + \sum\limits_{n=2}^\infty \frac{(i\varepsilon)^n}{n!}u^{(n)}(x) \cdot (Hv(x))^n.
	$$
A similar expression holds for $v$. Write
	\begin{multline}
		{\Re}\left((au+i\lambda v)(x+i\varepsilon Hv(x))\right) = au(x) - 
		\varepsilon\lambda v'(x)Hv(x) + \\ 
		+ a \sum_{k=1}^\infty \frac{(-1)^k \varepsilon^{2k}}{(2k)!} u^{(2k)}(x) (Hv(x))^{2k} -
		\lambda  \sum_{k=1}^\infty \frac{(-1)^k \varepsilon^{2k+1}}{(2k+1)!} v^{(2k+1)}(x) (Hv(x))^{(2k+1)}.
	\end{multline}
	
We have 
$$
\displaystyle \int \limits_{\RR} \varepsilon\lambda v'(x)Hv(x) =  |\varepsilon||\la\|v\|^2_{\onehalf},
$$
and this is the main term in \eqref{chg-var-ineq}.

We next estimate the contribution of the remaining terms repeatedly using
\begin{lemma}\label{h-norm-comp} If $f\in H_2$ and $\psi\in L_\infty$ satisfies $\psi'\in L_\infty$, then 
\begin{equation}\label{h12-h1}
	\|f \cdot \psi\|_\onehalf \le  \big\|\psi\|_{L_\infty} \cdot \big\|f\|_{L_2}+
	\big\|\psi\|_{L_\infty} \cdot \big\|f\|_{H_2} + \big\|\psi'\|_{L_\infty} \cdot \big\|f\|_{H_1}.
\end{equation}
\end{lemma}
\begin{proof} Write
$$
\|f \cdot \psi\|_\onehalf \leq \big\|f \cdot \psi\|_{L_2} + \big\|f\cdot \psi\|_{H^1},
$$
and that we have $\|f \cdot \psi\|_{L_2} \le \big\|\psi\|_{L_\infty} \cdot \big\|f\|_{L_2}$ and
$$\|f\cdot \psi\|_{H^1} \le \big\|\psi\|_{L_\infty} \cdot \big\|f'\|_{L_2} + \big\|\psi'\|_{L_\infty} 
\cdot \big\|f\|_{L_2}.$$
\end{proof}
 For brevity, we  say that a term is bounded by a constant if it admits an upper bound  
 depending only on  the constants \eqref{constants}.
 We start by noting that 
\begin{multline}\label{half-norm-est}
\|v'\cdot Hv\|^2_{\onehalf} \leq \|v'\cdot Hv\|^2_{L_2}+ \|v'\cdot Hv\|^2_{H_1}\leq \\ \leq
\|Hv\|^2_{L_{\infty}}\|v\|^2_{H_1} + 2\|v'\|^2_{L_{\infty}}\|v\|^2_{H_1}+2\|Hv\|^2_{L_{\infty}}\|v'\|^2_{H_1},
\end{multline}
whence the norm $\|v'\cdot Hv\|^2_{\onehalf}$ is bounded above by a constant.

We proceed to estimating the remaining terms.
If $\varepsilon$ satisfies \eqref{linf-hf}, the it follows directly  from  \eqref{fact-est}, \eqref{linf-hf} that there exists a constant $C_{10}>0$ depending only on $C_{00}$ such that
\begin{multline}
	\displaystyle \int_\R \sum_{k=1}^\infty \frac{\eps^{2k}}{(2k)!} \cdot |u^{(2k)}(x)| \cdot |Hv(x)|^{2k} dx < C_{10}; \\
	\displaystyle \int_\R \sum_{k=1}^\infty \frac{\eps^{2k+1}}{(2k+1)!} \cdot |v^{(2k+1)}(x)| \cdot |Hv(x)|^{2k+1} dx < C_{10}.
\end{multline}
Again using \eqref{h12-h1}, note that we have
\begin{multline*}
	\|u^{(k)}\cdot (Hv)^k\|_\onehalf \le \\ \le
	2 \left( \big\|Hv\|_{L_\infty}^k \cdot (\|u^{(k)}\|_{L_2} + \big\|u^{(k+1)}\|_{L_2})  + k \cdot \big\|Hv\|_{L_\infty}^{k-1} \cdot \big\|Hv'\|_{L_\infty} \cdot \big\|u^{(k)}\|_{L_2} \right);
\end{multline*}
\begin{multline*}
	\|v^{(k)}\cdot (Hv)^k\|_\onehalf \le \\ \le
	2 \left( \big\|Hv\|_{L_\infty}^k \cdot (\|v^{(k)}\|_{L_2} + \big\|v^{(k+1)}\|_{L_2})  + k \cdot \big\|Hv\|_{L_\infty}^{k-1} \cdot \big\|Hv'\|_{L_\infty} \cdot \big\|v^{(k)}\|_{L_2} \right),
\end{multline*}
whence the Sobolev $1/2$-seminorms
$$
\|\sum_{k=1}^\infty \frac{\eps^{2k}}{(2k)!} \cdot |u^{(2k)}(x)| \cdot |Hv(x)|^{2k}\|_{\onehalf}, 
$$
$$
\|\sum_{k=1}^\infty \frac{\eps^{2k+1}}{(2k+1)!} \cdot |v^{(2k+1)}(x)| \cdot |Hv(x)|^{2k+1}\|_{\onehalf}
$$
are bounded above by a constant provided $\varepsilon$ satisfies \eqref{linf-hf}.
The estimate \eqref{changed-moment} is proved.  Lemma \ref{linf-hf-norm-der} is proved completely.
\end{proof}

We conclude the proof of \eqref{chg-var-ineq}. By Lemma \ref{lem-bi-mult} and Corollary \ref{psitwo-q-cont-lp} there exist constants $\varepsilon_0$, $p_0>1$  and 
$C_{21}$ such that for any $p\in(1, p_0)$,   $\varepsilon\in(1, \varepsilon_0)$ we have
$$
\expsin \Psi[2, \varepsilon Hv]^p \leq C_{21}\exp (\varepsilon^2 \|Hv\|^2_{\BB(1/2)})\leq 
C_{22}\exp (\varepsilon^2 (\|v\|^2_{H(1)}+ 11\|Hv'\|^2_{\BB(1; \mathbb Z)}).
$$
Furthermore,  for any $p>0$, $\varepsilon>0$ we have 
$$
\expsin \Psi_{1+i \varepsilon Hv'}^p\leq  \exp(\varepsilon^2 p\|v\|^2_{H_1}). 
$$
The H\"older inequality and \eqref{changed-moment} imply \eqref{chg-var-ineq}.
\end{proof}

\subsection{Characteristic functions of additive statistics.}
In this subsection, the exponential decay estimate is formulated for the characteristic function of the joint distribution 
of additive statistics of the sine-process. The proof relies on a change of variable that reduces rapidly oscillating integrals to  exponentially decaying ones.  
  
  Let $H$ be a fixed horizontal strip symmetric with respect to the real axis. 
Take  a natural number $N$, a collection of 
functions $$\vec f=(f_1, \dots, f_{N}), f_1, \dots, f_{N}\in \onehalf,$$  
analytic in  $H$, two real vectors $\vec a=(a_1,\dots; a_{N})$, 
$\vec \la=(\la_1,\dots; \la_{N})$ and form the function 
$$
f(\vec a, \vec \la)=\sum\limits_{l=1}^{N} (a_l+i\la_l) f_l.
$$ 
Assume that 
$$
\displaystyle \int\limits_{\mathbb R}f_j(x)dx=0, j=1, \dots, N. 
$$
Denote 
$$
M_1(\vec f)=\max\limits_{j,k=1, \dots, N, j\ne k} |\langle f_j, f_k\rangle_{\onehalf}|;
$$
$$
M_2(\vec f)=\max\limits_{j=1, \dots, N} \|Hf_j\|_{L_{\infty}}.
$$ 
$$
M_3(\vec f)=\max\limits_{j=1, \dots, N} \|f_j\|_{\BH(1)}.
$$ 
$$
M_4(\vec f)=\max\limits_{j=1, \dots, N} \|f_j\|_{H_1}.
$$
$$M_5(\vec f)=\max\limits_{j=1, \dots, N}
(\|f_{j}^{(k)}\|_{L_{\infty}}+\|f_{j}^{(k)}\|_{L_{1}})/k!$$
Write 
$$
\sigma_l=\|f_l\|_{\onehalf},  l=1, \dots, N.
$$
\begin{proposition}\label{chg-var-est-prop}
There exist  constants $\gamma_0>0, D>0$  continuously and monotonically depending only on  
$M_1(\vec f), M_2(\vec f), M_3(\vec f), M_4(\vec f), M_5(\vec f)$ such that for any  $\gamma\in (0, \gamma_0)$ and any
real vectors $\vec a,\vec \la$ we have 
\begin{multline}\label{chg-var-multi-f}
\left|\expsin \exp S_{f(\vec a, \vec \la)}\right|\leq  \exp\Big(\displaystyle\sum\limits_{l=1} ^{N} (D+a_l^2)\sigma_l^2-
\displaystyle \frac{\gamma \displaystyle \sum\limits_{l=1} ^m  \la_l^2 \sigma_l^2 \ \ }
{\sqrt{\sum\limits_{l=1} ^m\la_l^2}\ \ }
+\gamma^2D(a^2+\la^2+1) \Big)
\end{multline}
\end{proposition}

\section{Conclusion of the proof of Theorem \ref{mainthm}}

\subsection{Conclusion of the proof of Lemma \ref{indep-lemma-highfreq}}

 We set $N=m+l$ conclude the proof of Lemma \ref{indep-lemma-highfreq}  applying 
 Corollary \ref{chg-var-est-prop}
 to the $N$ functions $F^{T,A},f^{d_1,A,r}, f^{d_2,A,s}$, $r=1, \dots, m-1, s=1, \dots, l$.
 
 We have already obtained estimates on the corresponding $\onehalf$-norms and inner products in 
 Propositions \ref{norm-fdal}, \ref{cov-fdal}.
 
 We proceed by establishing the uniform boundedness of the discrete norms 
 $
 \|(f^{d,A, l})'\|_{\BB(1)}.
 $
 \begin{proposition}\label{bbnorm-fdal}
For any $m\in \mathbb N$  there  exists  a positive constant $C$ such that 
for any $T\in \mathbb N$, any  $l=1, \dots, m-1$, any natural $d\in [T,  2T]$  
we have
\begin{equation}
\|(f^{d,A, l})'\|_{\BB(1)} \leq C.
\end{equation}
\end{proposition}
 \begin{proof}
 By definition,   the Fourier transforms $\widehat (f^{d,A, l})'$ are bounded and square-integrable, and also have bounded square-integrable derivatives. 
 It follows that the quantities  
 $$
 (1+|t|)f^{d,A, l}(t)
 $$
are uniformly bounded in $d,l,t$, whence also the desired uniform bound on the discrete norms 
 $$
 \|(f^{d,A, l})'\|_{\BB(1)}.
 $$
 \end{proof}
 
The uniform boundedness of  the 
Hilbert transforms of our functions has been checked in Proposition \ref{hilbtr-fdal}.
We proceed to a uniform bound on  the derivatives.

\begin{proposition}\label{unif-bdd-der}
For any $m\in \mathbb N$  there  exists  a positive constant $C$ such that 
for any $T\in \mathbb N$, any  $l=1, \dots, m-1$, any natural $d_1, d_2\in [T,  2T]$ satisfying $d_1-d_2>T^{l/m}$, 
we have
\begin{equation}
\|  HF_+^{T}\|_{H_1} \leq C;
\end{equation}
\begin{equation}
\|(Hf_+^{d_1,A,l})'\|_{L_{\infty}} \leq C, \|  (HF_+^{T})'\|_{H_1} \leq C;
\end{equation}
\begin{equation}
\|(f_+^{d_1,A,l})'\|_{L_{\infty}} \leq C, \|  (F_+^{T})'\|_{H_1} \leq C.
\end{equation}
\end{proposition}

\begin{proof}

We check that the derivatives $(f^{d, A, r})'$, $(Hf^{d, A, r})'$ are uniformly bounded.
To do so, note that the functions $\widehat{(f^{d, A, r})'}$, $\widehat{(Hf^{d, A, r})'}$ belong both to $L_1$ 
and $L_{\infty}$, have  their $L_1$ and $L_{\infty}$ norms  uniformly bounded both in $T$ and in $A\in (a_0, A_0)$, are 
differentiable  and the derivatives also belong both to $L_1$ and $L_{\infty}$, have  their $L_1$ and $L_{\infty}$ norms  uniformly bounded both in $T$ and in $A\in (a_0, A_0)$. We therefore obtain an estimate
$$
(1+|t|)|(f^{d, A})'(t)|, (1+|t|)|Hf^{d, A})'(t)|=O_{A,T}(1)
$$
and
$$
(1+|t|)|(f^{d, A, r})'(t)|, (1+|t|)|Hf^{d, A, r})'(t)|=O_{A,T, m}(1),
$$
as desired.
The estimate for the $\BB(1)$ follows immediately: we have 
$$
\|(f^{d, A})'(t)\|_{\BB(1)}, \|(Hf^{d, A})'(t)\|_{\BB(1)}=O_{A,T}(1)
$$
and
$$
\|(f^{d, A,r})'(t)\|_{\BB(1)}, \|(Hf^{d, A,r})'(t)\|_{\BB(1)}=O_{A,T,m}(1).
$$

\end{proof}

\subsection{Large deviations for subnormal random variables.}

Take random variables $Z_1, \dots, Z_m$ admitting exponential moments of all orders, 
and assume that the joint exponential moment of $Z_1, \dots, Z_m$ exhibits a certain similarity  
to that of independent Gaussian variables of variances $\sigma_1^2, \dots, \sigma_m^2$. For 
$\theta_1>0, \dots, \theta_m>0$, we
aim to obtain bounds for the event 
\begin{equation}\label{devia}
\{Z_1>\theta_1\sigma_1^2, \dots, Z_m>\theta_m\sigma_m^2\}.
\end{equation}
Let 
\begin{equation}\label{def-psi-z}
\psi_{\vec Z}(\la_1,...,\la_m) = \mathbb E \exp(\sum\limits_{l=1}^m \la_l Z_l)
\end{equation}
be the joint exponential moment of $Z_1, \dots, Z_m$. 
As before, we let $\norm (t; \sigma^2)$ be given by \eqref{gaussian-tail}.

\begin{lemma}\label{subnorm-lower}
For any  natural $m$, any constants $\theta_1>0, \dots, \theta_m>0$, $A_0, B_0, B_1>0$, $B_2>1+
\max\limits_{l=1, \dots, m}a_l^2$, there exist constants $D>0, B>0, \sigma>0$ such that 
the following holds for all  $\sigma_1>\sigma, \dots, \sigma_m>\sigma$.
 Let $Z_1,...,Z_m$ be random variables admitting exponential moments of all orders, whose joint exponential 
 moment \eqref{def-psi-z} satisfies:
 \begin{enumerate}\item
$$
\psi_{\vec Z}(\theta_1+iu_1, ..., \theta_m+iu_m)=
c_{\vec \theta}(u_1, ..., u_m) \exp(\frac 12\sum\limits_{l=1}^m \sigma_l^2(\theta_l+iu_l)^2),
$$
where $c_{\vec \theta}$ is a smooth function  such that $c_{\vec \theta}(\theta_1, \dots, \theta_m)\in (0,1)$ and 
\begin{equation}\label{grad-est-up}
  \|\mathrm{grad} \ c_{\vec \theta}(\theta_1+iu_1, \dots, \theta_m+iu_m)\|<A_0
  \end{equation}
  for all  $u_1,\dots,u_m$ satisfying $\max\limits_{l=1,...,m}\|u_m\|<B_0$.
\item for all  $u_1,\dots,u_m$ satisfying $\max\limits_{l=1,...,m}\|u_m\|\geq B_0$, we have
$$
\left|\psi_{\vec Z}(\theta_1+iu_1, ..., \theta_m+iu_m)\right|\leq \exp(B_1-B_2\sum\limits_{l=1}^m  \sigma_l^2).
$$
\end{enumerate}
Then 
\begin{multline}
|\displaystyle \int\limits_0^1\dots \displaystyle \int\limits_0^1 du_1\dots du_m \Prob(\{Z_1>\theta_1\sigma_1^2+u_1, 
\dots, Z_m>\theta_m\sigma_m^2+u_m\}) -\\- c_{\vec \theta}(\theta_1, ..., \theta_m) \displaystyle \int\limits_0^1\dots \displaystyle \int\limits_0^1 du_1\dots du_m 
\prod\limits_{l=1}^m {\mathscr N}(\theta_l\sigma_l^2+u_l; \sigma_l^2)| \leq \\ \leq  D(\sigma_1^{-1}+\dots +\sigma_m^{-1})
\frac{\exp(-\frac 12\sum\limits_{l=1}^m \theta_l^2\sigma_l^2)}{\sigma_1\dots \sigma_m}+D\exp(-B\sum\limits_{l=1}^n  \sigma_l^2).
\end{multline}
\end{lemma}
\begin{proof}
Let 
$$\psi_{\mathscr N}(\la_1, \dots, \la_m)=\exp(\frac12(\sum\limits_{l=1}^m \la_l^2\sigma_l^2))
$$
be the joint exponential moment of $m$ independent Gaussian random variables with expectation  $0$ 
and variances  $\sigma_l^2$, $l=1,\dots,m$.
For any $c>0$, the Parceval identity gives
\begin{multline}\label{parceval}
\displaystyle \int\limits_0^1\dots \displaystyle \int\limits_0^1  (\Prob(\{Z_1>\theta_1\sigma_1^2+u_1, 
\dots, Z_m>\theta_m\sigma_m^2+u_m\}) - c\prod\limits_{l=1}^m {\mathscr N}(\theta_l\sigma_l^2+u_l; \sigma_l^2))du_1\dots du_m=
\\=
\frac{\exp(-\sum\limits_{l=1}^m \theta_l^2\sigma_l^2)}{(2\pi)^m} \displaystyle \int\limits_{\RR}\dots \displaystyle \int\limits_{\RR} 
(\psi_{\vec Z}(\theta_1+iu_1,\dots, \theta_m+iu_m)-c\psi_{\mathscr N}(\theta_1+iu_1, \dots,\theta_m+iu_m))\times \\ \times 
\frac{\exp(-iu_l \theta_l\sigma_l^2)(\exp(-\theta_l-iu_l)-1) }{(\theta_l+iu_l)^2}du_1\dots du_m.
\end{multline}
We set $c=c_{\vec \theta}(\theta_1, \dots, \theta_m)$ and estimate the integrand in the right-hand side of \eqref{parceval} in absolute value.
We start by integrating over the set  $$\max\limits_{l=1,...,m}\|u_m\|<B_0.$$ We bound the multiple
$$
\left|\frac{\exp(-iu_l \theta_l\sigma_l^2)(\exp(-\theta_l-iu_l)-1) }{(\theta_l+iu_l)^2}\right|
$$
above by a constant depending only on $\theta_1, \dots, \theta_m$. Using we bound \eqref{grad-est-up}, we next estimate 
\begin{multline}
|\psi_{\vec Z}(\theta_1+iu_1,\dots, \theta_m+iu_m)-c\psi_{\mathscr N}(\theta_1+iu_1, \dots,\theta_m+iu_m)|\leq  \\ \leq
\exp(\frac 12\sum\limits_{l=1}^m \theta_l^2\sigma_l^2 )
A_0(\sum\limits_{l=1}^m |u_l|)\exp(-\frac 12\sum\limits_{l=1}^m \sigma_l^2u_l^2).
\end{multline}
Noting that 
$$
\displaystyle \int\limits_{\RR}\dots \displaystyle \int\limits_{\RR} 
(\sum\limits_{l=1}^m |u_l|)\exp(-\frac 12\sum\limits_{l=1}^m \sigma_l^2u_l^2)du_1\dots du_m\leq  
\frac{D(\sigma_1^{-1}+\dots +\sigma_m^{-1})}{\sigma_1\dots \sigma_m},
$$
we arrive at the desired estimate 
\begin{multline}
 \displaystyle \int\limits_{\{(u_1, \dots, u_m): \max\limits_{l=1,...,m}\|u_m\|<B_0\}} 
|\psi_{\vec Z}(\theta_1+iu_1,\dots, \theta_m+iu_m)-c\psi_{\mathscr N}(\theta_1+iu_1, \dots,\theta_m+iu_m))\times \\ \times 
\frac{\exp(-iu_l \theta_l\sigma_l^2)(\exp(-\theta_l-iu_l)-1) }{(\theta_l+iu_l)^2}du_1\dots du_m|\leq \\ \leq  D(\sigma_1^{-1}+\dots +\sigma_m^{-1})
\frac{\exp(\frac 12\sum\limits_{l=1}^m \theta_l^2\sigma_l^2)}{\sigma_1\dots \sigma_m}.
\end{multline}
In view of our second assumption, there exist constants $C_0, D_0$ such that  we have 
 \begin{multline}\displaystyle \int\limits_{{\{(u_1, \dots, u_m): \max\limits_{l=1,...,m}\|u_m\|\geq B_0\}} 
} 
\left|\psi_{\vec Z}(\theta_1+iu_1,\dots, \theta_m+iu_m)\right|+c\left|\psi_{\mathscr N}(\theta_1+iu_1, \dots,\theta_m+iu_m)\right|)\times \\ \times 
\frac{\exp(-iu_l \theta_l\sigma_l^2)(\exp(-\theta_l-iu_l)-1) }{(\theta_l+iu_l)^2}du_1\dots du_m\big|\leq \\ \leq C_0\exp(B_1-B_2
(\sum\limits_{l=1}^m \sigma_l^2))\int\limits_{\RR^m}\frac1{(\theta_l+iu_l)^2}du_1\dots du_m\leq D_0\exp(B_1-B_2
(\sum\limits_{l=1}^m \sigma_l^2)),
 \end{multline}
 and the proof is complete.
\end{proof}

For  any $\theta, \theta_r$, $r=1, \dots, m-1$ and  any $d\in [T, 2T]$ 
set
\begin{multline}
V(d, \theta, \theta_r)=\{X: S_{F^{T,A}}> \theta \log T,  
S_{f^{d_1,A,r}}\geq \theta_r \log T, r=1, \dots, m-1\}.
\end{multline}

For any $l\in \{1, \dots, m-1\}$,  any $\theta, \theta_r^{(1)}, \theta_s^{(2)}>0$,  $r=1, \dots, m-1$, 
$s=1, \dots, l$, and  any $d_1, d_2\in [T, 2T]$ satisfying
$$
T^{l/m}\leq |d_1-d_2|,
$$
set
\begin{multline}
V(d_1, d_2, \theta_r^{(1)}, \theta_s^{(2)})=\{X: S_{F^{T,A}}> \theta \log T,  \\
S_{f^{d_1,A,r}}\geq \theta_r^{(1)} \log T, r=1, \dots, m-1, 
S_{f^{d_2,A,s}}>\theta_s^{(2)} \log T, s=1, \dots, l\}.
\end{multline}

We have obtained
\begin{corollary}\label{indep-prob-cor-two}
For any natural $m\geq 3$ there exist constants $a_0, A_0$,  depending only on $m$, satisfying $1<a_0<A_0$, such that  for all $A\in (a_0, A_0)$ the following holds.
For any natural $m$ and any $\theta_0>0$ there exist constants $c>0, C>0, T_0>0$ such that  
for any $\theta, \theta_r^{(1)}, \theta_s^{(2)}\in (0, \theta_0)$, $r, s=1, \dots, m-1$, and   all  $T>T_0$ the following holds. 
\begin{enumerate} 
\item For any $d\in [T, 2T]$ we have
\begin{multline}\label{indep-prob-ineq-one}
c\leq \frac{\probsin(V(d, \theta, \theta_r)}
{\norm(\theta\log T, 2\log T) \prod\limits_{r=1}^{m-1} \norm (\theta_r\log T, \frac{2\log T}m)}\leq C. 
\end{multline}
\item
For any $l\in \{1, \dots, m-1\}$ and  any $d_1, d_2\in [T, 2T]$ satisfying
$$
T^{l/m}\leq |d_1-d_2|,
$$
we have 
\begin{multline}\label{indep-prob-ineq-two}
c\leq \frac{\probsin (V(d_1, d_2, \theta_r^{(1)}, \theta_s^{(2)}  ))}{
\norm(\theta\log T, 2\log T) \prod\limits_{r=1}^{m-1} \norm (\theta_r^{(1)} \log T, \frac{2\log T}m) 
\prod\limits_{s=1}^{l} \norm (\theta_s^{(2)} \log T, \frac{2\log T}m)}\leq C.
\end{multline}
\end{enumerate} 
\end{corollary}

We have also obtained an upper bound for the probability of large deviations for $S_{f^{d,A,m}}$.
\begin{corollary}\label{fm-large} For any natural $m\geq 3$ there exist constants $a_0, A_0$,  depending only on $m$, satisfying $1<a_0<A_0$, such that  for all $A\in (a_0, A_0)$ the following holds.
For any $d\in [T, 2T]$ we have
\begin{equation}\label{fm-large-ineq}
\probsin(\{X: S_{f^{d,A,m}}\geq \theta \log T/m \})
\leq  C\exp(-\frac{\theta^2\log T}{8m}).
\end{equation}
\end{corollary}

 Lemma \ref{main-est-vl} directly follows from Corollary \ref{indep-prob-cor-two},  Lemma \ref{main-est-w} from Corollary \ref{fm-large}.

\subsection{The occurrence of rare events for hierarchically independent random variables}\label{prob-lem-sec}

We derive Lemma \ref{main-lemma-tris} from Lemmata \ref{main-est-w}, \ref{main-est-vl}. Following the approach of Kistler \cite{kistler}, Arguin - Belius - Bourgade \cite{arguin}, we prepare

\begin{lemma}\label{hier-ind}
 
For any positive constants $C_1$, $C_2$, $C_3$ there exist positive constants $\gamma>0$, $\delta>0$, $N_0$ such that 
the following holds for all $n>N_0$.
Assume that  we are given a positive function  $\rho$ on $\NN$ satisfying $\rho(n)>n^{-1}$, and, for any natural $n$,  a positive function $\varphi_n$ on 
$\mathbb N$  such  that 
\begin{equation}\label{phi-rho}
\sum\limits_{k=1}^{n}  \varphi_n(k)<C_1 n\rho(n).
\end{equation}
If $(\Omega, \mathcal F, \Prob)$ is a probability space, and  $V_1, \dots, V_n$ events satisfying 
\begin{enumerate}
\item $C_2^{-1}\rho(n)\leq \Prob(V_l)<C_2\rho(n);$
 \item $\Prob(V_k\cap V_j)<C_3\rho(n)\varphi(|k-j|)$,
\end{enumerate}
then we have
$$
\Prob  \left(\left\{\sum\limits_{l=1}^n {\mathbb I}_{V_l}>\gamma n\rho(n)\right\}\right)>\delta.
$$
\end{lemma}
\begin{proof}
Indeed, by our first assumption we have 
$$
C_0^{-1}<\frac{\ee (\sum\limits_{l=1}^n {\mathbb I}_{V_l})}{ n\rho(n)}<C_0.
$$
Next, we have \begin{multline}
\ee (\sum\limits_{l=1}^n {\mathbb I}_{V_l})^2=\sum\limits_{k,l=1}^n \Prob({V_k\cap V_l})\leq \\
\leq \sum\limits_{k=1}^n \Prob({V_k})+ \sum\limits_{k,l: |k-l|\geq 1} \Prob({V_k\cap V_l})\leq \\ \leq 
 C_1n\rho(n)+ n\rho(n)\sum\limits_{l=1}^n \varphi(k)\leq (1+C_1+C_2)(n\rho(n))^2.
\end{multline} 
In the last inequality  we used the assumption $\rho(n)>n^{-1}$.
The Lemma  follows now from the Paley-Zygmund inequality. 
\end{proof}
 Lemmata \ref{main-est-w}, \ref{main-est-vl}, \ref{hier-ind}  directly imply Lemma \ref{main-lemma-tris}.
Lemma \ref{main-lemma-tris} is proved. Lemma \ref{main-lemma} is proved. 
Theorem \ref{mainthm} is proved completely. \qed

\end{document}